\documentclass{article}

\usepackage{amsmath}
\usepackage{amsthm}
\usepackage{amsfonts}
\usepackage{amssymb}
\usepackage{graphicx}
\usepackage{caption}
\usepackage{subcaption}
\usepackage{setspace}
\usepackage{epstopdf}
\usepackage{authblk}

\usepackage{tikz}
\usetikzlibrary{arrows.meta, shapes.geometric}
 \usepackage{hyperref}
\usepackage{cleveref} 
\usepackage{wrapfig}
\usepackage{float}
\usepackage{algorithm2e}

\usepackage{fullpage}

\theoremstyle{definition}

\newtheorem{remark}{Remark}[section]
\newtheorem{proposition}{Proposition}[section]

\newcommand{\MC}{{M\kern-0.20em C}}

\newcommand{\x}{{\bf x}}

\newcommand{\y}{{\bf y}}

\newcommand{\edit}[1]{{#1}}

 \RestyleAlgo{ruled}

\begin{document}

\title{Stochastic Quadrature Rules for Solving PDEs using Neural Networks}
\author[1]{Jamie M. Taylor (\texttt{jamie.taylor@cunef.edu}) }
\author[2,3,4]{David Pardo (\texttt{david.pardo@ehu.eus})}
\affil[1]{Department of Mathematics, CUNEF Universidad, Madrid, Spain }
\affil[2]{University of the Basque Country (UPV/EHU), Leioa, Spain}
\affil[3]{Basque Center for Applied Mathematics (BCAM), Bilbao, Spain}
\affil[4]{Ikerbasque: Basque Foundation for Science, Bilbao, Spain}
\date{}
\maketitle

\abstract{
We examine the challenges associated with numerical integration when applying Neural Networks to solve Partial Differential Equations (PDEs). We specifically investigate the Deep Ritz Method (DRM), chosen for its practical applicability and known sensitivity to integration inaccuracies. Our research demonstrates that both standard deterministic integration techniques and biased stochastic quadrature methods can lead to incorrect solutions. In contrast, employing high-order, unbiased stochastic quadrature rules defined on integration meshes in low dimensions is shown to significantly enhance convergence rates at a comparable computational expense with respect to low-order methods like Monte Carlo. Additionally, we introduce novel stochastic quadrature approaches designed for triangular and tetrahedral mesh elements, offering increased adaptability for handling complex geometric domains. We highlight that the variance inherent in the stochastic gradient acts as a bottleneck for convergence. Furthermore, we observe that for gradient-based optimisation, the crucial factor is the accurate integration of the gradient, rather than just minimizing the quadrature error of the loss function itself.

\section{Introduction}

Neural Networks (NNs) are potent tools for solving certain Partial Differential Equations (PDEs) \cite{blechschmidt2021three,huang2022partial,karniadakis2021physics,raissi2019physics},   for example, when dealing with high-dimensional problems \cite{han2018solving,zang2020weak}, inverse problems \cite{guo2022monte,khoo2019switchnet,kim2024review}, and parametric problems \cite{baharlouei2025least,brevis2021machine,brevis2024learning,khoo2021solving,uriarte2022finite,wang2021learning}. In particular, they are able to overcome the curse of dimensionality faced by classical methods such as Finite Element Methods (FEM) and Finite Difference Methods.

When solving a single PDE, the trial function is approximated by an NN, and approximating a solution to the PDE reduces to defining an appropriate loss function to be minimised over the trainable parameters that describe the NN. There exist multiple loss functions that can be employed for solving PDEs using NNs, including Physics Informed Neural Networks (PINNs) \cite{cai2021physics,cai2021physics2,raissi2019physics}, Variational Physics Informed Neural Networks (VPINNs) \cite{kharazmi2019variational,kharazmi2021hp,rojas2024robust,
taylor2024adaptive,taylor2023deep} and the Deep Ritz Method (DRM) \cite{uriarte2023deep,wang2022variational,yu2018deep}. They all employ a loss function defined in terms of integrals that must be evaluated via numerical quadrature. Here, we focus on designing adequate quadrature rules for PINNs, VPINNs, and DRM.

The standard approach to integration when employing NNs is to use Monte Carlo methods, which are simple to implement and offer dimension-independent scalings of the variance in terms of the number of quadrature points. Importance sampling - based on sampling nodes from non-uniform distributions designed to reduce the variance - can further increase accuracy. These approaches are used in \cite{nabian2021efficient}, where authors  employ a piecewise-constant approximation of the PINN loss integrand to define an importance sampling distribution, with a similar mesh-based sampling scheme for PINNs employed in \cite{yang2023dmis}. In \cite{wan2024adaptive}, adaptive importance sampling methods are used to train an auxiliary generative NN to produce integration points in the Deep Ritz Method, whilst \cite{nabian2021efficient} aims to use the loss function itself as the sampling distribution. The addition of collocation points where pointwise residuals are large may be used to improve convergence \cite{gao2023failure}. We highlight, however, that all of these methods focus on accurately integrating the loss rather than the gradient used by the optimiser. 

In the context of parametric PDEs for inverse problems, often the parameter space is high-dimensional, whilst the computational domain remains low dimensional (typically one-, two- or three-dimensional). In these problems, more accurate integration strategies than Monte Carlo over the computational domain may be employed. One approach that has produced successful results is interpolated neural networks, which project the NN onto a classical finite element space. They leverage the flexibility of the NN structure with the well-understood linear theory of finite elements and, in particular, employ exact integration over the computational domain \cite{badia2024compatible,badia2024finite,badia2025adaptive, li2025unfitted,rivera2022quadrature,omella2024r}.

In classical numerical methods that employ piecewise-polynomial trial functions such as Finite Element Methods (FEM) and Isogeometric Analysis (IGA), exact quadrature rules may be obtained. The most classical example is the $n$-point Gaussian quadrature rules in one dimension that are capable of integrating exactly degree $2n-1$ polynomials (see, for example, \cite{alma991000509117308131}). In higher-regularity spaces such as those employed by IGA, more efficient quadrature rules, such as Greville quadrature, may be employed, maintaining exactness at a lower cost (see \cite{auricchio2012simple,hughes2010efficient,johannessen2017optimal} and references therein). When employing NNs, however, no exact quadrature methods are known. The use of a fixed, deterministic quadrature schemes may lead to overfitting at the quadrature nodes, yielding small values of the residual at those locations whilst generalising poorly to the full computational domain \cite{rivera2022quadrature}. The use of a stochastic Monte Carlo scheme mitigates overfitting at the expense of lower accuracy than Gauss-type rules for fixed integrands. In this work, we will focus on stochastic Gauss-type rules, which overcome overfitting by employing new quadrature rules at each iteration whilst producing significantly more accurate quadrature in low-dimensional problems. 

The main contribution of this article is twofold: first, we explore and analyse the advantages and limitations of Gauss-type, stochastic integration techniques when using NNs to solve low-dimensional PDEs. In particular, we show how deterministic integration methods can lead to catastrophic results, and that stochastic methods can overcome overfitting issues and ensure convergence in an appropriate limit. We highlight the issue of bias in stochastic quadrature, which may lead to erroneous results, especially in VPINNs and DRM, even when employing an arbitrarily fine integration scheme. All of this is illustrated with simple numerical examples. In the second main contribution and based on the aforementioned observations, we construct novel stochastic, unbiased Gauss-type quadrature rules on triangular and tetrahedral elements, which offer flexibility when producing integration meshes in complex geometries. These new quadrature rules are designed by taking advantage of the symmetry of the integration elements so that integrating the given polynomial space reduces to a single-parameter problem. Throughout, we take the Deep Ritz Method as our prototypical problem as it is a simple test case of a loss function based on a single integral that is highly sensitive to integration errors. This sensitivity is expected in other losses based on weak formulations, such as WANs and VPINNs,\footnote{PINNs exhibit less sensitivity to integration errors because the loss function satisfies \edit{a property analogous to} the interpolation property \cite[Definition 4.9.]{garrigos2023handbook}, which we comment on in \Cref{subsec.DRM}.} but we avoid complications related to the choice of test function spaces. Finally, the DRM is a problem of practical interest due to its consistency with the energy-norm error. 

The article is organised as follows. \Cref{subsec.DRM} outlines the Deep Ritz Method (DRM) and our model problems and architectures. \Cref{secDeterministic} highlights how the DRM may fail catastrophically when deterministic integration is used. Motivated by this, \Cref{subsec.quad.theory} introduces key properties related to stochastic quadrature methods, focusing on Gauss-type rules, and outlines the optimisation strategy we employ in the sequel. \Cref{secBias} identifies how biased integration methods may lead to incorrect results. In \Cref{secUnbiasedRules}, we review existing unbiased, Gauss-type stochastic rules from the literature and propose novel rules based on triangular and tetrahedral elements, empirically measuring their variances on simple integrands to understand their behaviour in the pre-asymptotic regime. \Cref{secDRMResults} presents numerical results with the DRM to illustrate how integration errors may affect convergence and highlight how variance reduction - equivalent to mitigating integration errors for unbiased rules - may significantly improve convergence. We further see in this section how convergence is limited by the variance of the stochastic gradient rather than the variance of the loss. Finally, we make concluding remarks in \Cref{secConclusions}, with technical results deferred to \Cref{appTriTet,appErrorEstimates}.

\section{Model Problems Using the Deep Ritz Method (DRM)}\label{subsec.DRM}

In this work, we consider Poisson's equation with homogeneous Dirichlet boundary conditions as our prototypical elliptic PDE. That is, we aim to find $u^*:\Omega\to\mathbb{R}$ such that 
\begin{equation}
\begin{array}{rc l}
\Delta u^*(\x)=& f(\x) &\forall\x\in\Omega,\\
u^*(\x)=&0& \forall\x\in\partial\Omega. 
\end{array}
\end{equation}
For simplicity, we take $\Omega=[0,1]^d$ for $d=1,2$, or $3$ and $f$ is chosen
such that $u^*$ is given by the manufactured solutions 
\begin{equation}\label{eq.exact.sol}
u^*(\x)=10\left(|\x-\x_0|^2-\frac{1}{16}\right)\prod\limits_{i=1}^d\sin(\pi x_i),
\end{equation}
where $\x_0$ is the centre of $\Omega$, namely, $\frac{1}{2}$ in 1D, $\left(\frac{1}{2},\frac{1}{2}\right)$ in 2D, and $\left(\frac{1}{2},\frac{1}{2},\frac{1}{2}\right)$ in 3D. We highlight that these are smooth solutions with few features.

\paragraph{The Deep Ritz Method.}

The continuum loss function employed in this work is that of the Deep Ritz Method, $\mathcal{L}_C:H^1_0(\Omega)\to\mathbb{R}$, given by
$$
\mathcal{L}_C(u)=\int_\Omega \frac{1}{2}|\nabla u(\x)|^2+f(\x)u(\x)\,d\x,
$$
whose unique minimiser in $H^1_0(\Omega)$, $u^*$, satisfies Poisson's equation in weak form, 
$$
\int_\Omega \nabla u^*(\x)\cdot\nabla v(\x)+f(\x)v(\x)\,d\x=0
$$
for all $v\in H^1_0(\Omega)$. The loss function is related to the $H^1_0$-error by the following equation:
\begin{equation}\label{eq.loss.equiv.error}
\mathcal{L}_C(u)-\mathcal{L}_C(u^*)=\frac{1}{2}\int_\Omega |\nabla u(\x)-\nabla u^*(\x)|^2\,d\edit{\x}, 
\end{equation}
for all $u\in H^1_0(\Omega)$. Thus, a loss reduction is proportional to a decrease in the square of the $H^1_0$ \edit{error}. The minimal value of the loss is generally strictly negative and unknown {\it a priori}. 

The DRM considers the trial function to be an NN, so that $u=u(\x;\theta)$, where $\theta\in\Theta$ are the trainable parameters of the NN. The boundary conditions may be imposed via a penalty method or strongly by imposing them into the architecture. Herein, we take the latter approach. The continuum problem is thus approximated as 
$$
\min\limits_{u\in H^1_0(\Omega)}\mathcal{L}_C(u)\approx \min\limits_{\theta\in \Theta}\mathcal{L}_C(u(\cdot;\theta)).
$$ 

\paragraph{Architectures.}\label{subsec.Model.Architecture}

We employ an NN architecture that automatically enforces the boundary conditions. Our architecture corresponds to a classical, feed-forward, fully-connected NN with a cutoff function. Explicitly, we define $\tilde{u}(x;\theta)$ to be a fully-connected, feed-forward NN with 3 hidden layers of 30 neurons each using $\tanh$ as our activation function, and define $u$ as
\begin{equation}
u(\edit{\x};\theta)=\tilde{u}(\edit{\x};\theta)\left(c_d\prod\limits_{i=1}^d x_i(1-x_i)\right), 
\end{equation}
where $c_d$ is a constant depending on the dimension $d$ to ensure that the $H^1_0$-norm of the cutoff function is one. This yields 1,951 trainable parameters in 1D, 1,981 in 2D, and 2,011 in 3D.

\paragraph{Discretisation of the loss function.}
The loss function is then discretised via a quadrature rule and implemented into a gradient-based optimisation scheme using Automatic Differentiation (AD) to evaluate the gradient. Explicitly, we employ a quadrature rule $\xi=(\x_j,w_j)_{j=1}^J$, where $(\x_j)_{j=1}^J$ are the quadrature nodes and $(w_j)_{j=1}^J$ are the weights, to approximate $\mathcal{L}_C$, yielding a discretised loss function $L$, so that 
$$
\mathcal{L}_C(u(\cdot;\theta))\approx L(\theta;\xi):=\sum\limits_{j=1}^J \left(\frac{1}{2}|\nabla u(\x_j;\theta)|^2+f(\x_j)u(\theta;\x_j)\right)w_j. 
$$
Since the gradient $\frac{\partial L}{\partial\theta}(\theta;\xi)$ is evaluated via AD of the loss function, it inherits the quadrature rule used to evaluate the loss, i.e., 
$$
\frac{\partial L}{\partial\theta}(\theta;\xi)=\sum\limits_{j=1}^J \left(\frac{\partial}{\partial\theta}\left(\frac{1}{2}|\nabla u(\x_j;\theta)|^2+f(\x_j)u(\theta;\x_j)\right)\right)w_j. 
$$
In particular, this corresponds to the integral of a high-dimensional, vectorial function with the same rule used for every component. 

The aforementioned model problem employs a simple loss function where the integrand is generally signed and non-zero at the exact minimiser, which makes integration more challenging. In \edit{constrast, the PINNs loss is analogous to losses admitting the} {\it interpolation} property (see \cite[Definition 4.9.]{garrigos2023handbook}), which states that the loss may be decomposed as a sum of functions
$$
\mathcal{L}(\theta)=\sum\limits_{i}f_i(\theta)
$$
such that there is a common value $\theta^*$ which minimises each $f_i$, that is, 
$$
f_i(\theta^*)=\inf f_i.
$$
In the case of PINNs, each $f_i$ would correspond to the residual at point $\x_i\in\Omega$. This structure, in which pointwise minimisation is equivalent to global minimisation, guarantees stability at the exact solution even in the presence of integration errors and exhibits automatic variance reduction during training \cite{ma2018power}. 

\edit{In the VPINNs and WANs methods, each loss function requires the evaluation of integrals corresponding to the weak form of the PDE, 
$$
\int_\Omega \nabla u(\x;\theta)\cdot \nabla \varphi(\x)+\varphi(\x)f(\x)\,d\x, 
$$ 
where VPINNs employs a family of user specified test functions $\varphi\in H^1_0(\Omega)$ whilst WANs employs $\varphi$ as an adversarial NN within a min-max optimisation problem. Such integrals arise as the first variation of the DRM loss function, so whilst they are based on different principles, their structures are related and, thus, expected to inheret the same challenges as DRM with respect to integration. However, the VPINNs and WANs methodologies are significantly more complex, employing a large, user-specified test space in the former and a delicate adversarial training regime in the latter. As we aim to study the effect of integration alone, the DRM provides an ideal test case.}

\section{Deterministic quadrature}
\label{secDeterministic}

When employing deterministic quadrature rules, the DRM loss is highly susceptible to catastrophic overfitting, as noted in \cite{rivera2022quadrature}. If the discretised loss is taken as 
$$
L(\theta)=\sum\limits_{j=1}^J \left(\frac{1}{2}|\nabla u(\x_j;\theta)|^2+f(\x_j)u(\x_j;\theta)\right)w_j,
$$
with $\x_j\in\Omega$ and $w_j\geq 0$ fixed across all iterations, then $L(\theta)$ is generally unbounded from below. To see this, consider $u$ to have a small gradient at the finite number of quadrature nodes, and a large value of the opposite sign to $f$ at the nodes. Whilst this implies large gradients between the quadrature nodes, they are unseen by the discretised loss function. Due to the lack of lower bound, any accurate optimisation scheme leads to values of the discretised loss $L(\theta)\ll \mathcal{L}_C(u^*)$ and the introduction of erroneous, sharp transitions between quadrature nodes. 

\Cref{fig.DRM.det} illustrates this issue. We consider the 1D model problem and the NN architecture outlined in \Cref{subsec.Model.Architecture} and a variant of the Adam optimiser that we detail in \Cref{subsecAdam}. We train for 4,000 iterations with a midpoint rule using 20 uniformly-spaced quadrature points. We show the approximate solution and its derivative in \Cref{subfig_det_u,subfig_det_du}, respectively, highlighting the quadrature nodes' locations. Large jumps in the solution are observed away from quadrature nodes, particularly near the first and second. Nonetheless, the derivative attains small values at the nodes themselves, with in-between regions exhibiting large derivatives that are unseen by the loss. \Cref{subfig_det_loss} shows how the discretised loss achieves values far below the exact minimum of the continuum loss, which may only happen in the presence of significant integration errors. The simplest manner to avoid catastrophic overfitting is to employ a stochastic quadrature rule that employs new samples at every iteration.

\begin{figure}[H]\begin{center}
\begin{subfigure}[t]{0.45\textwidth}
\begin{center}
\includegraphics[width=0.95\textwidth]{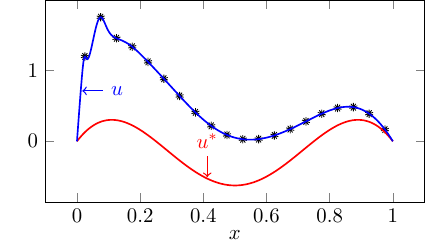}
\caption{The exact solution $u^*$ and the approximated solution $u$.  Black points indicate the location of the quadrature points. \label{subfig_det_u}}
\end{center}
\end{subfigure}\hspace{0.05\textwidth}
\begin{subfigure}[t]{0.45\textwidth}
\begin{center}
\includegraphics[width=0.95\textwidth]{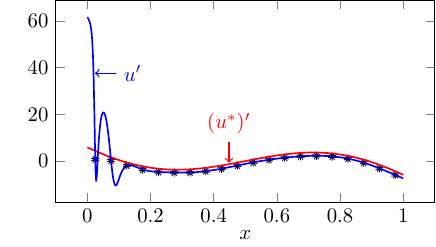}
\caption{The derivatives of the exact solution $u^*$.  Black points indicate the location of the quadrature points.\label{subfig_det_du}}
\end{center}
\end{subfigure}\hspace{0.05\textwidth}
\begin{subfigure}{0.45\textwidth}
\begin{center}
\includegraphics[width=0.95\textwidth]{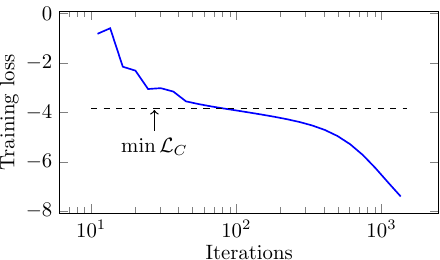}
\caption{Discretised loss $\edit{L}(\theta)$ per iteration. The dashed line shows the exact minimum of the continuous loss.\label{subfig_det_loss}}
\end{center}
\end{subfigure}
\caption{Results of the DRM using deterministic quadrature.}\label{fig.DRM.det}
\end{center}
\end{figure}

\section{Stochastic quadrature}\label{subsec.quad.theory}

\subsection{Monte Carlo methods}

The simplest stochastic integration rule is a vanilla Monte Carlo (MC) method, where the integral of an integrable function $f:\Omega\to\mathbb{R}$ on a bounded domain $\Omega\subset\mathbb{R}^d$ is approximated via 
$$
Q(f;(\x_n)_{n=1}^N)=\frac{|\Omega|}{N}\sum\limits_{n=1}^N f(\x_n),
$$
with $\x_n$ taken as independent and identically distributed (i.i.d.) uniform variables in $\Omega$. The expected value and variance of this quadrature rule are:
\begin{equation}
\begin{split}
\mathbb{E}(Q(f;(\x_n)_{n=1}^N))=& \int_\Omega f(\x)\,d\x,\\
\text{Var}(Q(f;(\x_n)_{n=1}^N))=&\frac{|\Omega|}{N}\left(\frac{1}{|\Omega|}\int_\Omega f(\x)^2\,d\edit{\x}-\left(\frac{1}{|\Omega|}\int_\Omega f(\x)\,d\x\right)^2\right),
\end{split}
\end{equation}
with infinite variance if $f$ is not square-integrable. When the variance is finite, it admits a scaling of $O(N^{-1})$, converging to zero as $N\to \infty$, independently of the dimension $d$. However, there is no {\it uniform} bound on the integration error as $N\to \infty$. That is, for every $N\in \mathbb{N}$, $\epsilon>0$ and any continuous integrand $f\in C(\bar{\Omega})$, the probability that
$$
\left|Q(f;(\x_n)_{n=1}^N)-\int_\Omega f(\x)\,d\x\right|\geq\max\limits_{\x_0\in\Omega}\left||\Omega|f(\x_0)-\int_\Omega f(\x)\,d\x\right|-\epsilon
$$
is positive. In particular, whilst being improbable, large integration errors are possible even for large $N$.

Quasi-Monte Carlo methods aim to reduce variance by taking samples that are more uniformly distributed in space, such as Sobol points. In such cases, \edit{the quadrature rule becomes deterministic, whilst errors for sufficiently regular integrands admit scalings of order $O\left(\frac{\ln(N)^d}{N}\right)$}, where $d$ is the dimension, requiring $N\gg 2^d$ points to ensure better asymptotic scalings than a classical Monte Carlo method. This was shown to improve convergence in the DRM in \cite{chen2019quasi}. Another approach is to use importance sampling, where the quadrature nodes are taken as i.i.d. variables from a non-uniform probability distribution $\mathcal{P}$ on $\Omega$, with the quadrature rule taken as 
$$
Q(f;(\x_n)_{n=1}^N)=\frac{1}{N}\sum\limits_{n=1}^N \frac{f(\x_n)}{\mathcal{P}(\x_n)}. 
$$
In this case, the expected value of $Q$ is the exact integral, whilst the variance now has the form 
$$
\text{Var}(f;(\x_n)_{n=1}^N) =\frac{1}{N}\left(\int_\Omega \frac{f(\x)^2}{\mathcal{P}(\x)}\,d\x-\left(\int_\Omega f(\x)\,d\x\right)^2\right). 
$$
Whilst the asymptotic scaling of $O(N^{-1})$ is the same as in a classical MC scheme, \edit{corresponding to the choice of $\mathcal{P}(\x)=\frac{1}{|\Omega|}$, an appropriate selection} of $\mathcal{P}$ may significantly reduce the multiplicative constant. If $f$ is strictly positive, then taking $\mathcal{P}(\x)$ proportional to $f$ yields zero variance. 

An alternative approach is to use stratified Monte Carlo methods. The integration domain may be partitioned into elements $(E_n)_{n=1}^N$, with the stratified Monte Carlo rule corresponding to a summation of Monte Carlo methods on each element of the domain. \edit{In this work, we will consider the simplest case: a single, uniformly sampled quadrature point per element according to a uniform distribution, $\x_n\sim \mathcal{U}(E_n)$. Thus, the corresponding stratified MC rule yields 
\begin{equation}\label{eqStratifiedMC}
Q(f;(\x_n)_{n=1}^N)=\sum\limits_{n=1}^N f(\x_n)|E_n|. 
\end{equation}
Such an approach is exact for piecewise-constant functions over the partition $(E_n)_{n=1}^N$. In particular, we interpret the rule as a stochastic Gauss-type rule that is exact for piecewise-constants on the integration mesh described by the partition $(E_n)_{n=1}^N$.

This idea may be mixed with importance sampling approaches both within elements and between elements, and allows for using larger number of quadrature points in elements that admit higher variances in the MISER method \cite{press1990recursive}. The simplest such case arises by fixing a total number of quadrature points $N$ and considering a given partition of the domain into two elements, $\Omega=E_1\cup E_2$ and taking the quadrature rule as the sum of independent Monte Carlo integration in each subdomain using $N_1$ and $N_2=N-N_1$ points in $E_1$ and $E_2$, respectively. By allocating more points to the higher variance region, the total variance may be reduced, and the optimal allocation admits the simple formula of $N_1=\frac{N\sigma_1}{\sigma_1+\sigma_2}$, where $\sigma_1$ and $\sigma_2$ are the standard deviations of the Monte Carlo rules employed to integrate $f$ over $E_1$ and $E_2$, respectively. This approach necessarily reduces the variance compared to an equal weighting $N_a=N_b$, unless $\sigma_a=\sigma_b$, in which case it remains unchanged. In \eqref{eqStratifiedMC}, we are envisioning a large number of elements with few integration points per element; nonetheless, the same variance reduction ideas can be applied as with the simple case of just two elements.
}

\subsection{Gauss-type stochastic quadrature on integration meshes}
We will employ the notation $A\lesssim B$ to indicate that there exists some $C>0$, which may depend on the quadrature rule and its domain, but not on the integrand, such that $A\leq CB$. Similarly, we will use the notation $A\sim B$ to mean that $A\lesssim B$ and $B\lesssim A$ simultaneously. The technical statements and hypotheses of this section are deferred to \Cref{appErrorEstimates}.

To construct stochastic, Gauss-type rules, we first consider a bounded master element $\Omega_0\subset\mathbb{R}^d$, and consider random samples $\xi=(\x_j,w_j)_{j=1}^J \in (\Omega_0\times\mathbb{R})^J$ of quadrature points and weights taken according to a probability measure $\mu$ over $(\Omega_0\times\mathbb{R})^J$. We define the corresponding quadrature rule on our master element for an integrable function $f$ as 
$$
Q_M(f;\xi)=\sum\limits_{j=1}^J f(\x_j)w_j. 
$$
For $p\in\mathbb{N}$, we say that the stochastic quadrature rule is {\it of order $p$} if 
$$
Q_M(f;\xi)=\int_{\Omega_0} f(\x)\,d\x
$$
for every polynomial $f$ of order at most $p$ with $(\mu-)$probability 1. In particular, if $f\in C^{p+1}(\overline{\Omega}_0)$, $Q_M$ is of order $p$, and the weights $w_j$ are non-negative with probability 1, then 
\begin{equation}\label{eqMasterElementPBound}
\begin{split}
\left|Q_M(f;\xi)-\int_{\Omega_0}f(x)\,dx\right|\lesssim ||\nabla^{p+1}f||_\infty
\end{split}
\end{equation} 
Furthermore, we say that a rule is {\it unbiased} if 
$$
\mathbb{E}(Q_M(f;\xi))=\int_{\Omega_0}f(\x)\,d\x
$$
for all integrable $f:\Omega_0\to\mathbb{R}$.

Our global quadrature rules, inherited from $Q_M$, are then described via a partition of a computational domain $\Omega$ into elements $(E_n)_{n=1}^{N_e}$, so that $E_n=\y_n+A_n\Omega_0$ for a $d\times d$ matrix $A_n$ with positive determinant and $\y_n\in\mathbb{R}^d$. We take the elements to be open, pairwise disjoint, and such that $$
\bar{\Omega} = \bigcup\limits_{n=1}^{N_e}\overline{E}_n,
$$
as illustrated in \Cref{figIntegrationMesh}. We define the element-wise quadrature rules, $Q_n$, by 
$$
Q_n(f;\xi_n)=\sum\limits_{j=1}^J f(\y_n+A_n\edit{\x_{n,j}})\edit{w_{n,j}}\det(A_n), 
$$
with \edit{$\xi_n=(\x_{n,j},w_{n,j})_{j=1}^J$} sampled according to $\mu$ \edit{for each $n=1,...N$}, which then yield the global quadrature rule 
$$
Q(f;(\xi_n)_{n=1}^{N_e})=\sum\limits_{n=1}^{N_e}Q_n(f;\xi_n). 
$$

\begin{figure}[h!]
\begin{center}
\begin{tikzpicture}

\node[draw,thick, regular polygon, regular polygon sides=3, minimum size=3cm] (omega) at (0,0) {};
\node at (omega.center) {\large $\Omega_0$};

\begin{scope}[shift={(4,-0.5)}]
    \fill[blue!30] (0,0) -- (1,0) -- (1,1) -- cycle;
    \fill[blue!30] (0,1) -- (1,2) -- (0,2)-- cycle;

    \draw[thick] (0,0) -- (2,0) -- (2,2) -- (0,2) -- cycle;
    \draw[thick] (0,0) -- (2,2);
    \draw[thick] (1,0) -- (2,1);
    \draw[thick] (0,1) -- (1,2);
    \draw[thick] (1,0) -- (1,2);
    \draw[thick] (0,1) -- (2,1);

    \node at (0.7,0.3) {$E_1$};
    
    \node at (0.3,1.7) {$E_2$};

\end{scope}

\node at (2,0.9) {$\y_2 + A_2\x$};

\node at (2.2,-0.25) {$\y_1 + A_1\x$};

\draw[-{Stealth[length=3mm, width=3mm]}, thick] (omega.east) -- node[midway, below] {} (4.5,0.);
\draw[-{Stealth[length=3mm, width=3mm]}, thick] (omega.east) -- node[pos=0.3, above] {} (4.,1.25);

\end{tikzpicture}
\end{center}
\caption{Illustration of the master element and a mesh-based quadrature rule.\label{figIntegrationMesh}}
\end{figure}
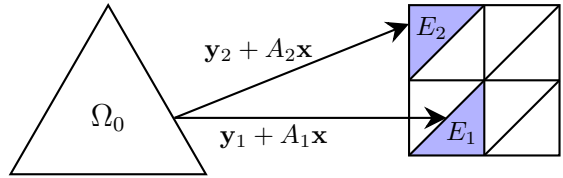

Unless stated otherwise, we take each $\xi_n$ to be i.i.d. according to $\mu$. When unambiguous, we will denote the global quadrature rule by $Q(f)$.

If $Q_M$ is unbiased, then the resulting quadrature rule on $\Omega$ is also unbiased, so that 
\begin{equation}
\mathbb{E}\left(Q(f;(\xi_n)_{n=1}^{N_e})\right)=\int_\Omega f(\edit{\x})\,d\edit{\x}. 
\end{equation}
Under certain geometric restrictions on the integration elements (see \Cref{propVarUpperBound}), we obtain uniform quadrature error estimates. If $Q_M$ is an order-$p$ rule with positive weights, or more generally satisfies an inequality of the form \eqref{eqMasterElementPBound}, $Q$ will satisfy, with probability 1,
\begin{equation}\label{eqUniformPError}
\begin{split}
\left|\int_\Omega f(\edit{\x})\,d\edit{\x} - Q(f,(\xi_n)_{n=1}^{N_e})\right| \lesssim N^{-\frac{p+1}{d}}||\nabla^{p+1}f||_\infty
\end{split}
\end{equation}
for all $f\in C^{p+1}(\bar{\Omega})$, where $N=JN_e$ is the total number of quadrature points. 

Furthermore, if the samples $\xi_n$ are i.i.d. variables, under the same geometric restrictions of the partition, when $Q_M$ is unbiased and satisfies an inequality of the form \eqref{eqMasterElementPBound}, we obtain the following estimation of the variance: 
\begin{equation}\label{eqVarEst}
\begin{split}
\text{Var}(Q(f;(\xi_n)_{n=1}^{N_e}))\lesssim &||\nabla^{p+1}f||_{\infty}^2N^{-1-\frac{2p+2}{d}}. 
\end{split}
\end{equation}

We see that the asymptotic variance estimate \eqref{eqVarEst} yields a better asymptotic rate than MC, regardless of $d\geq 1$ and $p\geq 0$, although this advantage diminishes in high dimensions. In the pre-asymptotic regime, however, the analysis becomes more delicate and, thus, the benefits are uncertain (see \Cref{subsecOrderFixed} for a detailed numerical study). Furthermore, in higher dimensions, the number of integration elements in a uniform partition grows exponentially with $d$, so employing an integration mesh may require a prohibitive number of integration elements even when considering a rather coarse mesh.

\begin{remark}
In this work, we consider only the cases where the mappings from the master element $\Omega_0$ to the integration elements $E_n$ are affine. Nonetheless, if a nonlinear mapping $\varphi_n:\Omega_0\to E_n$ is employed, one may similarly define an element-wise quadrature rule, 
$$
Q_n(f;\xi)=\sum\limits_{j=1}^J f(\varphi_n(\x_j))\det(\nabla \varphi_n(\x_j))w_j. 
$$
If the quadrature rule is unbiased, then, provided sufficient regularity of $\varphi$ (e.g., $\varphi_n\in C^1(\overline{\Omega}_0)$ and $\det(\nabla\varphi(\x))$ is bounded away from zero), one has that $f(\varphi_n(\x))\det(\nabla \varphi_n(\x))$ is continuous, and via a change of variables one obtains
$$
\mathbb{E}(Q_n(f;\xi)) = \int_{\Omega_0}f(\varphi_n(\x))\det(\nabla \varphi_n(\x))\,d\x=\int_{E_n}f(\x)\,d\x.
$$
However, if the quadrature rule is of order $p$ on the master element, the quadrature rule $Q_n$ will not be of order $p$, but rather it will be exact if $\tilde{f}(\x):=f(\varphi_n(\x))\det(\nabla\varphi_n(\x))$ is an order-$p$ polynomial. In particular, for $f\in C^{p+1}(\bar{\Omega})$ and $\varphi_n\in C^{p+2}(\Omega_0)$, one may reclaim similar estimates involving the element-wise deformations $\varphi_n$, but this is beyond the scope of this work. 
\end{remark}

\subsection{Stochastic Gradient-based Methods}

When employing stochastic quadrature rules, we are approximating our loss function by a random variable, $L(\theta;\xi)$, where $\theta\in\Theta$ are the trainable weights of our network and $\xi$ corresponds to a sample of quadrature points and weights according to a probability measure $\mu$. Stochastic gradient-based optimisers update the parameters $\theta$ via the stochastic gradient, such as the SGD update rule,
$$
\theta_{k+1}=\theta_k-\gamma_k\frac{\partial L}{\partial\theta}(\theta_k;\xi), 
$$ 
where $\gamma_k>0$ are iteration-dependent learning rates. Whilst a deterministic gradient descent, under sufficient regularity assumptions such as strong convexity, would exhibit linear convergence to their minima, the variance in the stochastic gradient generally impedes convergence to a particular point, at best leading to the iterates tending to a steady-state probability distribution (see, for example, \cite{liu2021noise}). As such, to guarantee convergence, one must consider regimes with the learning rates $\gamma_k$ tending to zero adequately. A series of works (see, for example, \cite{barakat2021convergence,garrigos2023handbook,jin2023continuous,li2022revisiting,yu2021analysis}, and references therein) have considered such regimes rigorously for various gradient-based optimisers, and whilst the precise technical assumptions and conclusions vary across works, the common conclusion is that if the learning rates and other internal hyperparameters satisfy appropriate decay conditions, then the iterates converge to a critical point $\theta^*$ of the expected value of the stochastic loss, 
$$\mathcal{L}(\theta):=\mathbb{E}(L(\theta;\xi))=\int_XL(\theta;\xi)\,d\mu(\xi),$$ in the style of a central-limit theorem, that
\begin{equation}\label{eqCLTSGD}
\frac{\theta_{k+1}-\theta^*}{\sqrt{\gamma_k}}\to \mathcal{N}(0,W).
\end{equation}
Here, $\mathcal{N}$ denotes a multivariate Normal distribution and $W$ is a matrix that typically scales linearly in the covariance matrix of the stochastic gradient,
$$
\text{Cov}\left(\frac{\partial L}{\partial \theta}(\theta;\xi)\right)=\mathbb{E}\left(\frac{\partial L}{\partial \theta}(\theta;\xi)\otimes \frac{\partial L}{\partial \theta}(\theta;\xi)\right)-\mathbb{E}\left(\frac{\partial L}{\partial \theta}(\theta;\xi)\right)\otimes \mathbb{E}\left(\frac{\partial L}{\partial \theta}(\theta;\xi)\right).
$$
These results require that the stochastic gradient be an unbiased estimator of the gradient of $\mathcal{L}$, so that
$$
\frac{\partial \mathcal{L}}{\partial\theta}(\theta)=\mathbb{E}\left(\frac{\partial L}{\partial\theta}(\theta;\xi)\right). 
$$
In particular, we identify two key aspects: First, the expected values of the stochastic loss and gradient play key roles in identifying the long-term behaviour of the optimisation scheme, and this is intimately linked to the need for unbiased rules when using stochastic quadrature. Second, once the quadrature rule is unbiased, decreasing the variance of the stochastic gradient reduces the integration errors and improves the convergence rate.

\label{subsecAdam}
Within this work, we will employ a particular variant of the Adam optimiser, proposed in \cite{barakat2021convergence}, which we detail in Algorithm \ref{alg.Adam}.

\begin{algorithm}\caption{Adam Variant (\cite{barakat2021convergence}). All operations are considered componentwise. }\label{alg.Adam}
\KwData{Step sizes $(\gamma_k)$, first order moment decay rates $(\beta^1_k)$, second order moment decay rates $(\beta^2_k)$, regularisation constant $\epsilon$, i.i.d. samples $(\xi_k)$, stochastic loss function $(L)$}
{\bf Initialisation}: $\theta_0\in\Theta$, $m_0=0$, $v_0=0$, $r^1_0=0$, $r^2_0=0$\;
\For{$k=0,...,k_{\max}$}{
$m_{k+1}=\beta^1_km_k+(1-\beta^1_k)\frac{\partial L}{\partial \theta}(\theta_k,\xi_k)$;\\
$v_{k+1}=\beta^2_kv_k+(1-\beta^2_k)\frac{\partial L}{\partial \theta}(\theta_k,\xi_k)^2$;\\
$r^1_{k+1}=\beta^1_kr^1_k+(1-\beta^1_k)$;\\
$r^2_{k+1}=\beta^2_kr^2_k+(1-\beta^2_k)$;\\
$\hat{m}_{k+1}=m_{k+1}/r^1_{k+1}$;\\
$\hat{v}_{k+1}=v_{k+1}/r^2_{k+1}$;\\
$\theta_{k+1}=\theta_k-\gamma_k\frac{\hat{m}_{k+1}}{\epsilon+\sqrt{\hat{v}_{k+1}}}$;
}
\end{algorithm}

\cite[Theorems 5.2, 5.7]{barakat2021convergence} shows that, under specific technical assumptions on the loss function and chosen hyperparameters, the iterates $\theta_k$ converge in probability to a critical point of $\mathcal{L}$. Furthermore, given that the iterates converge, then they satisfy a Central Limit Theorem-type result of the form \eqref{eqCLTSGD}, where $W$ is a matrix depending on the statistical properties of $L$ and its gradient, the local behaviour of $\mathcal{L}$ near the critical point, and the chosen hyperparameters in the optimiser. We do not detail here their technical assumptions required on the stochastic loss $L$, as these cannot be readily verified and many, such as strict convexity, are not expected to hold when using NNs.  

The covariance $W$ admits a semi-explicit expression that highlights two factors that may deteriorate convergence: (a) a large covariance in the stochastic gradient, and (b) ill-conditioning of the Hessian of $\mathcal{L}$. Reducing the former is the focus of this work. The latter, dictated by the structure of $\mathcal{L}$ and the NN architecture, is significantly more difficult to control. We highlight the work of \cite{de2023operator} that considered a theoretical treatment of the commonly encountered issue of ill-conditioning in physics-informed machine learning and proposed potential avenues to overcome them; however, they currently remain impractical for complex problems.

We employ the following choices of $\gamma_k,\beta^1_k,\beta^2_k$, and $\epsilon$, which are in agreement with Assumptions 5.1 and 5.3 in \cite{barakat2021convergence}. 
\begin{equation}\label{eqAdamHyperparamters}
\begin{split}
\gamma_k = &
\begin{cases} 
\gamma_0, & \text{if } k < \sqrt{k_{\max}}, \\
\frac{c}{b + k}, & \text{otherwise}.
\end{cases},\\
\beta^1_k =& 1-0.9\frac{\gamma_k}{\gamma_0},\\
\beta^2_k = &1-0.1\frac{\gamma_k}{\gamma_0},\\
\epsilon=&10^{-2},
\end{split}
\end{equation}
where $k$ is the current iteration, $k_{\max}$ is the total number of iterations employed for each particular experiment, \( \gamma_0 = 10^{-2} \) is the initial learning rate, and \( \gamma_f = 10^{-4} \) is the final learning rate. The constants $b$ and $c$ are determined so that $\frac{c}{b+\sqrt{k_{\max}}}=\gamma_0$, whilst $\gamma_{k_{\max}}=\gamma_f$, and are explicitly given by:\footnote{We emphasise that, whilst the work of \cite{barakat2021convergence} may, in certain cases, guarantee convergence in an asymptotic limit, as the user must prescribe the optimiser's hyperparameters at each iteration, this yields a significantly large numbers of degrees of freedom that, in practice, may not lead to acceptable solutions in a feasible number of iterations. Our choice of \eqref{eqAdamHyperparamters} was consequently {\it ad-hoc} and informed by preliminary studies.}

\begin{equation}
\begin{split}
b =& \frac{- (\gamma_f / \gamma_0) k_{\max} + \sqrt{k_{\max}}}{-1 + (\gamma_f / \gamma_0)},\\
c =& \gamma_0 \left( \frac{(\sqrt{k_{\max}} - k_{\max}) (\gamma_f / \gamma_0)}{-1 + (\gamma_f / \gamma_0)} \right).
\end{split}
\end{equation}
If the iterates converge as in \eqref{eqCLTSGD}, our choice of $\gamma_k$ implies convergence of the iterates of the form

\begin{equation}
||\theta_k-\theta^*||=O\left(\frac{1}{\sqrt{k}}\right).
\end{equation}
Furthermore, if the mapping from the parameter space $\Theta$ to the NN space $\{u(\cdot;\theta):\theta\in\Theta\}\subset H^1_0(\Omega)$ is (locally) Lipschitz, then:
\begin{equation}\label{eq.H1.converge}
||u(\cdot;\theta_k)-u(\cdot;\theta^*)||_{H^1_0(\Omega)}=O\left(\frac{1}{\sqrt{k}}\right). 
\end{equation}

\section{Bias}
\label{secBias}

When a sufficiently regular stochastic loss function $L(\theta;\xi)$ is implemented into an appropriate gradient-based optimiser, the best-case scenario is that the iterates $\theta_k$ converge to a (local) minimum $\theta^*$ of the average loss, 
$$
\mathcal{L}(\theta):=\mathbb{E}(L(\theta;\xi)). 
$$
Our aim is to minimise a continuum loss function $\mathcal{L}_C$ over an NN space, so to obtain good solutions, it is necessary that the stochastic loss is {\it unbiased}. In the case of the DRM, we require that 
\begin{equation}
\begin{split}
\int_\Omega \frac{1}{2}|\nabla u(\x)|^2+f(\x)u(\x)\,d\x=\mathbb{E}\left(\sum\limits_{n=1}^N \left( \frac{1}{2}|\nabla u(\x_n)|^2+f(\x_n)u(\x_n)\right)w_n\right).
\end{split}
\end{equation}
In this section, we show how biased integration rules may lead to erroneous results when using the DRM.  

First, we define a biased quadrature rule $Q^b_M(f)=w_1f(x_1)+w_2f(x_2)$ over the master element $[-1,1]$. We take $x_1\sim \mathcal{U}(0,1)$ and $x_2\sim \mathcal{U}(-1,0)$ independently. Then, the weights $w_1$ and $w_2$ are constructed so that the rule is order one, i.e., 

\begin{equation}
\begin{split}
w_{1} =\frac{2 x_{2}}{x_{2}-x_{1}};\hspace{0.025\textwidth} w_{2} =-\frac{2 x_{1}}{x_{2}-x_{1}}.
\end{split}
\end{equation}
To realise that the rule is biased, for $f\in C[-1,1]$, we evaluate the expected value of the quadrature rule as
\begin{equation}\label{eqP1Bias}
\begin{split}
\mathbb{E}(Q^b_M(f))=& \int_0^1\int_{-1}^0\frac{2 x_{2}}{x_{2}-x_{1}}f(x_1)\,dx_2\,dx_1+\int_0^1\int_{-1}^0-\frac{2 x_{1}}{x_{2}-x_{1}}f(x_2)\,dx_2\,dx_1\\
=& \int_0^1 f(x_1)\left(2+2x_1\ln\left(\frac{x_1}{1+x_1}\right)\right)\,dx_1+\int_{-1}^0 f(x_2)\left(2 x_{2} \ln \left(\frac{x_{2}-1}{x_{2}}\right) +2\right)\,dx_2\\
=&\int_{-1}^1 f(x)\underbrace{\left(2+2|x|\ln\left(\frac{|x|}{1+|x|}\right)\right)}_{W_0(x)}\,dx=\int_{-1}^1 f(x)W_0(x)\,dx. 
\end{split}
\end{equation}
The deviation of $W_0$ from 1 is indicative of bias. This will be compared with a two-point, first-order, unbiased rule, to be denoted $P_1$ in the sequel, following a similar construction but taking $x_2=-x_1$, rather than independent sampling.
We consider the 1D model problem with exact solution \eqref{eq.exact.sol}, with the architecture outlined in \Cref{subsec.Model.Architecture}, using 50,000 iterations of the Adam optimiser variant outlined in \Cref{subsecAdam}, employing the biased ($P^b_1$) and unbiased $(P_1)$ rules outlined previously. In both cases, we use a uniform partition of $[0,1]$ into 32 integration elements, giving 64 quadrature points for each rule.

\begin{figure}[H]\begin{center}
\begin{subfigure}[t]{0.45\textwidth}
\begin{center}
\includegraphics[width=0.95\textwidth]{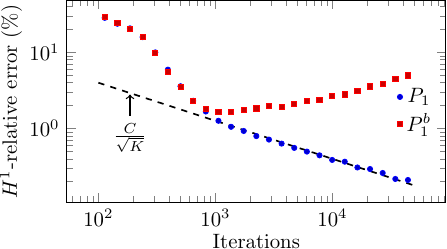}
\caption{The mean $H^1$ relative errors (\%) over logarithmically-spaced intervals. \label{fig_DRM_bias_h1} }
\end{center}
\end{subfigure}\hspace{0.05\textwidth}
\begin{subfigure}[t]{0.45\textwidth}
\begin{center}
\includegraphics[width=0.95\textwidth]{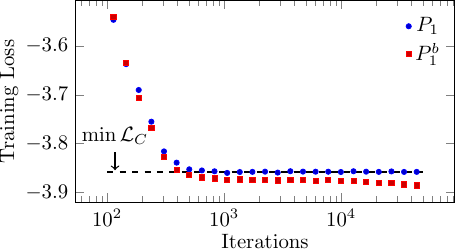}
\caption{The mean of the discretised loss over logarithmically-spaced intervals The dashed lines indicates the minimum of the continuum loss.\label{fig_DRM_bias_loss}}
\end{center}
\end{subfigure}
\begin{subfigure}[t]{0.45\textwidth}
\begin{center}
\includegraphics[width=0.95\textwidth]{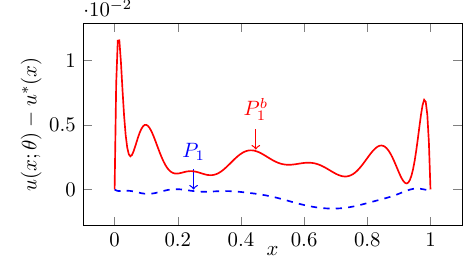}
\caption{The pointwise errors of the approximated solutions. \label{fig_DRM_bias_er}}
\end{center}
\end{subfigure}\hspace{0.05\textwidth}
\begin{subfigure}[t]{0.45\textwidth}
\begin{center}
\includegraphics[width=0.95\textwidth]{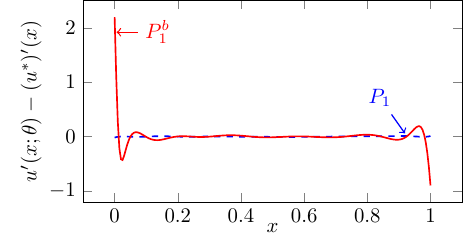}
\caption{The pointwise errors of the derivatives of the approximate solutions.\label{fig_DRM_bias_der}}
\end{center}
\end{subfigure}
\caption{Comparison of results using biased ($P_1^b$) and unbiased ($P_1$) order-one rules.}\label{fig.DRM.bias}
\end{center}
\end{figure}

\Cref{fig_DRM_bias_loss} presents the tendency of the loss via a log-binning approach\footnote{Here, and throughout the article, we employ log-binning approaches to show the tendency of data in the presence of noise \cite{clauset2009power}. Explicitly, for some $\beta>1$, the $x$-coordinates correspond to $\beta^{i+0.5}$ with $i\in\mathbb{N}$, whilst the $y$ coordinate corresponds to the average of the measured value for all iterations between $\beta^i$ and $\beta^{i+1}$.}. It shows that there is a significant integration issue when $P_1^b$ is employed, as the average values of the discrete loss fall below the exact minimum of $\mathcal{L}_C$. The increase of the error after several hundred iterations in \Cref{fig_DRM_bias_h1} when using $P_1^b$ is indicative of convergence to a distinct solution, and we see significant errors in the derivative of the approximated solution in \Cref{fig_DRM_bias_der}. The unbiased rule $P_1$, however, shows the correct convergence to the exact solution, attaining a final $H^1$-error of around $0.2\%$. \edit{In particular, \Cref{fig_DRM_bias_h1} shows that the asymptotic behaviour remains roughly parallel to a curve of the form $\frac{C}{\sqrt{k}}$, shown for comparison purposes.}

To understand the failure of $P_1^b$, we consider the bias of the global quadrature rule inherited from that on the master element when $N=h^{-1}$ uniform elements are employed. A change of variables on each element gives that 
$$
\mathbb{E}(Q(f))=\int_0^1 f(x)W_h(x)\,dx, 
$$
where $W_h(\x)$ is the $h$-periodic extension of $W_0\left(\frac{1}{2h}(x+1)\right)$. Thus, $W_h(x)$ admits oscillations of the same amplitude as those of $W_0(x)$, with frequency $h$ (see \Cref{fig.WPlot}). We recall that, even though the quadrature rule is biased, as it is order-one with positive weights, then we have a uniform quadrature error estimate 
$$
\left|Q(f)-\int_0^1 f(x)\,d\x\right|\lesssim h^{2}||f''||_\infty
$$
for $f\in C^{2}([0,1])$. In particular, for a {\it fixed} integrand, a sufficiently fine integration mesh can adequately approximate the integral. 
\\
\begin{figure}[h!]
\begin{center}
\includegraphics[width=0.45\textwidth]{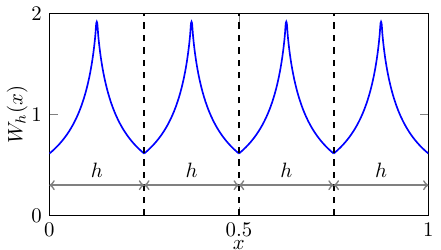}
\end{center}
\caption{Illustration of $W_h$ for $h=\frac{1}{4}$.}\label{fig.WPlot}
\end{figure}

Now, let us consider the effect of this quadrature rule when using the DRM. If we consider an arbitrary trial function $u\in H^1_0(0,1)$, then the expected value of the loss function is 
\begin{equation}
\int_0^1 \left(\frac{1}{2}u'(x)^2+f(x)u(x)\right)W_h(x)\,dx. 
\end{equation}
As $W_0$ is bounded away from $0$ and $\infty$, there exists a unique minimiser $u_h$, depending on $h$, satisfying the Euler-Lagrange equation, 
\begin{equation}
\int_0^1 W_h(x) u_h'(x) v'(x)+f(x)v(x)W_h(x)\,dx =0
\end{equation}
for all $v\in H^1_0(0,1)$. This may be identified as a case of periodic homogenisation (see, for example, \cite[Chapter 1]{bensoussan2011asymptotic}). Classical results give that the solutions $u_h$, as $h\to 0$, converge strongly in $L^2(0,1)$ and weakly in $H^1(0,1)$ to the unique solution $u_0\in H^1_0(0,1)$ of 
$$
\int_0^1 W_{\text{eff}} u_0'(x) v'(x)+f(x)v(x)\,dx=0
$$
for all $v\in H^1_0(0,1)$. Here, $W_{\text{eff}}$ is the harmonic mean of $W_0$ and, by Jensen's inequality, distinct from 1. Thus, outside of exceptional cases - for example, if $f$ is identically zero - we have that the minimisers of the average loss function, as $h\to 0$, converge strongly in $L^2$ to a distinct function, whilst it is not expected that they converge strongly in $H^1$ to anything. In higher dimensions, similar pathologies are expected, although the effective coefficient $W_{\text{eff}}$ may now be matrix-valued and lack an explicit representation.

As a result, when employing a biased quadrature rule in the DRM, its minimisers will generally not coincide with those of the continuous problem. In practice, this effect may be mitigated by the known spectral bias of NNs, whereby the NN fails to approximate high-frequency modes of the solution (see, for example, \cite{aldirany2024multi,cao2021towards,rahaman2019spectral}), as the NN would have to form high-frequency oscillations in order to approximate the oscillatory solutions $u_h$. In particular, with a very fine mesh, the effects of a bias/variance trade-off may yield satisfactory results with a biased integration rule. This failure is expected to be common to other loss functions based on weak formulations.

\section{Unbiased stochastic quadrature rules}\label{secUnbiasedRules}

\label{subsec.Properties.unbiased}

\subsection{Quadrature rules from the literature}\label{subsec.Old.quad}
\paragraph{Order-0 unbiased quadrature rule ($P_0$)}
As proposed in \cite{haber1966modified}, we consider the master element $\Omega_0=[0,1]^d$, and define the unbiased, one-point, order-$0$ rule $P_0$ by taking $\xi=(\x,1),$ uniformly distributed in $\Omega_0$, and defining 
$$
Q_M(f;\xi)=f(\x). 
$$
This rule is a stratified Monte Carlo method when extended to an integration mesh. Indeed, when partitioning $\Omega$ into elements $(E_n)_{n=1}^N$, choosing each $\x_n\in E_n$ uniformly, the global quadrature rule becomes 
$$
Q(f;\xi)=\sum\limits_{n=1}^{N_e}f(\x_n)|E_n|.
$$
\paragraph{Order-1 unbiased quadrature rule ($P_1$).}
As proposed in \cite{haber1967modified} we define the unbiased, two-point, order-one rule $P_1$ on the master element $\Omega_0=[-1,1]^d$ by taking $J=2$, $\x_1$ to be uniformly distributed on $[-1,1]^d$, and $\x_2=-\x_1$. The corresponding weights are $w_1=w_2=\frac{1}{2^{d-1}}$, so that 
$$
Q_M(f;\xi)=\frac{1}{2^{d-1}}\left(f(\x_1)+f(-\x_1)\right).
$$
This quadrature rule is exact for all linear functions {\it and, as a symmetric rule, is exact for all odd functions}.
\paragraph{Order-3 unbiased quadrature rule ($P_3$).}
We consider the master element $[-1,1]^d$. Following \cite{siegel1985unbiased}, we encounter different definitions depending on the dimension. In one dimension, the rule has three quadrature points. The value $x_1$ is taken according to the distribution $\mathcal{P}(x)=3x^2$ on $(0,1)$. The quadrature rule is then given by 
$$
Q_M(f;\xi)=\frac{f(x_1)-2f(0)+f(-x_1)}{3x_1^2}+2f(0). 
$$
Whilst the weight corresponding to the quadrature point $x=0$ may be positive or negative, \cite{siegel1985unbiased} shows that this quadrature rule maintains a uniform error estimate of the form \eqref{eqMasterElementPBound}. 

In two dimensions, a five-point, order-three rule is employed on the master element $[-1,1]^2$. The numbers $x_1,y_1$ are taken from the distribution $\mathcal{P}(x,y)=\frac{3}{2}(x^2+y^2)$ on $(0,1)^2$. Then, the quadrature rule is defined as 
$$
Q_M(f;\xi)=2\frac{f(x_1,y_1)+f(-y_1,x_1)+f(y_1,-x_1)+f(-x_1,-y_1)-4f(0,0)}{3(x_1^2+y_1^2)}+4f(0,0). 
$$
The authors of \cite{siegel1985unbiased} also outline an efficient procedure for sampling from $\mathcal{P}$. Again, whilst the weights may be negative, the authors obtain an estimate of the form \eqref{eqMasterElementPBound}.
Finally, on the master element $[-1,1]^3$, $P_3$ corresponds to the tensor product rule using the three-point rule in 1D and the five-point rule in 2D, yielding a 15-point, order-three rule. \edit{These rules are again symmetric, and thus trivially integrate all odd functions exactly}.

\subsection{Proposed quadrature rules}\label{subsec.New.Quad}

Herein, we further propose several unbiased rules for triangular master elements in 2D and tetrahedral master elements in 3D. Our interest in triangle- and tetrahedra-based meshes arises due to their flexibility when meshing complex geometries. Whilst these rules appear as extensions of the methodologies presented in \cite{haber1967modified,siegel1985unbiased}, to the authors' knowledge, they have not appeared in the literature. All of the proposed rules are built by exploiting the symmetries of the master elements.

\subsubsection{$P^\text{tri}_1$ - Order-1 rule in 2D triangular elements}\label{subsub.p1.2d}

The master element $T$, taken as an equilateral triangle, has vertices $$\left(0,1\right),\left(-\frac{\sqrt{3}}{2},-\frac{1}{2}\right), \left(\frac{\sqrt{3}}{2},-\frac{1}{2}\right). $$
We introduce the rotation matrix $R$, corresponding to a rotation by $\frac{2\pi}{3}$, for convenience. We remark that $T$ is invariant under $R$, so that $RT=T$. Furthermore, $R^3=I$, and $I+R+R^2=0$, where the latter will be essential for defining symmetric rules. Throughout, we will denote $\x=(x,y)\in T\subset\mathbb{R}^2$. 

\edit{The following construction is analogous to that of \cite{haber1967modified} and is based on symmetry. Linear functions, which integrate to zero, may be integrated exactly by a symmetric rule corresponding to points generated by repeated rotations by $R$ of the triangle with equal weights. The weight is then chosen to integrate constants exactly, and we finally see that by taking a uniform distribution, the rule is unbiased. We now detail the precise construction. }

We define the three-point rule $P^1_t$ via the following construction. Take $\x\in T$ according to a uniform distribution on $T$, and consider three quadrature nodes, 
$$
\x_i=R^i\x
$$
for $i=0,1,2$. The quadrature rule is then defined as 
$$
Q_M(f)=\frac{\sqrt{3}}{4}\sum\limits_{i=0}^2f(\x_i). 
$$
As $|T|=\frac{3\sqrt{3}}{4}$, it is immediate that $Q_M(f)$ is exact for constant functions. Furthermore, we can see that it is exact on linear functions, as in this case $f(\x)={\bf v}\cdot\x$, and thus 
\begin{equation}
\begin{split}
Q_M(f)=&\frac{\sqrt{3}}{4}\sum\limits_{i=1}^3{\bf v}\cdot R^i\x\\
=&\frac{\sqrt{3}}{4}{\bf v}\cdot\left((I+R+R^2)\x\right)=0, 
\end{split}
\end{equation}
corresponding to the exact value. Finally, we see that $P^1_t$ is unbiased, as 
\begin{equation}
\begin{split}
\mathbb{E}(Q_M(f))=&\frac{1}{|T|}\int_{T}\frac{\sqrt{3}}{4}\sum\limits_{i=0}^2f(R^i\x)\,d\x\\
=&\frac{|T|}{3|T|}\sum\limits_{i=0}^2\int_T f(R^i\x)\,d\x\\
=& \int_T f(\x)\,d\x, 
\end{split}
\end{equation}
as $T$ is invariant under the measure-preserving rotation $R$. 
Uniform sampling on the triangle is performed as in \cite[Chapter XI, Section 2]{devroye2006nonuniform} by considering $a_1$ and $a_2$ to be i.i.d. uniform variables on $[0,1]$, and then relabelled such that $a_1<a_2$. The point is then written in barycentric coordinates as 
$$
\x = a_1 {\bf v}_1+(a_2-a_1){\bf v}_2+(1-a_2) {\bf v}_3
$$
with ${\bf v}_i$ being the vertices of the reference triangle.

\subsubsection{$P_2^\text{tri}$ - Order-2 rule in 2D triangular elements}\label{subsub.p2.2d}

\edit{The following construction is analogous to that of \cite{siegel1985unbiased} and, like the $P_1^{\text{tri}}$ rule, is based on symmetry. By employing a symmetric rule of points generated by repeated rotations of a sampled point, linear functions are integrated exactly. The weights are then chosen to ensure that quadratic functions are integrated exactly, with the weight depending on the initially sampled point. In constrast to the $P_1^{\text{tri}}$ rule, a central point must then be added to ensure constant functions are integrated exactly, with the sampling distribution then selected to ensure that the rule is unbiased. We detail the precise construction below.  }

We define the following four-point rule

\begin{equation}
Q(f)=w_1\left(f(\x)+f(R\x)+f(R^2\x)-3f(0)\right)+\frac{3\sqrt{3}}{4}f(0), 
\end{equation}
where 
$$
w_1(\x)=\frac{\sqrt{3}}{16|\x|^2}. 
$$
Finally, we take $\x$ from the probability distribution on $T$,
$$\mathcal{P}(\x)=\frac{1}{3w_1(\x)}=\frac{16|\x|^2}{3\sqrt{3}}.$$

\Cref{app.P2.2d} shows that this is an order-2 unbiased rule. Sampling from $\mathcal{P}$ is straightforward via a rejection criterion by observing that it is the restriction of the radial distribution $\frac{2}{\pi}|\x|^2$ on the unit ball to $T$. Thus, we sample $\x = r(\cos(\theta),\sin(\theta))$ with $r=u^\frac{1}{4}$, where $u\sim U(0,1)$ and $\theta\sim U(0,2\pi)$, and reject the sample if it lies outside of $T$. This has an acceptance rate of $\frac{3\sqrt{3}}{4\pi}\approx 41\%$. As a rejection-based sampler provokes implementation difficulties within the NN optimisation loop, herein, we produce a database of 10,000 points sampled from this distribution, which are then sampled uniformly during training.

Whilst a finite database of integration points may lead to overfitting, as we employ a number of iterations of the order of tens of thousands, each rule will be seen per-element only a handful of times during training. Furthermore, if $N_e$ is the number of elements, this yields roughly $4\times 10^4 N_e$ possible distinct quadrature points, so we do not expect to see overfitting. 
The proposed quadrature rule admits negative weights, which are in fact unbounded from above at the non-zero quadrature nodes and from below at zero. Nonetheless, we show in \Cref{subsecUnifEstimatesP2} that the quadrature rule satisfies the following uniform estimate for $f\in C^3(\bar{T})$:
$$
\left|Q_M(f)-\int_T f(\x)\,d\x\right|\lesssim ||\nabla^{3}f||_\infty. 
$$
In particular, when extended to a uniform mesh of $N_e$ elements, denoted by $Q$, we have a global variance estimate

$$
\text{Var}\left(Q(f)\right)\lesssim N^{-4}||\nabla^3f||_\infty^2. 
$$
for $f\in C^3(\bar{\Omega})$, where $N$ is the number of quadrature points employed (see \Cref{propVarUpperBound}).

\subsubsection{$P^\text{tet}_1$ - Order-1 rule in 3D on tetrahedral elements}\label{subsub.p1.3d}

We consider a reference tetrahedron $T$ with vertices $(1,1,1),(-1,-1,1),(-1,1,-1)$ and $(1,-1,-1)$. For our four-point, order-one rule, \edit{we employ the same arguments as for the $P_1^{\text{tri}}$ rule in our construction.} We first introduce the rotation matrix 
$$
R_0=\begin{pmatrix}
0 & 0& -1\\ 0 & -1 & 0\\ 1 & 0 & 0
\end{pmatrix}.
$$
We remark that $R_0T=T$, so that $T$ is invariant under $R_0$. Furthermore, 
$$
I+R_0+R_0^2+R_0^3=0. 
$$
Thus, it is immediate that the quadrature rule 
$$
Q(f)=\frac{2}{3}\left(f(\x)+f(R_0\x)+f(R_0^2\x)+f(R_0^3\x)\right)
$$
is exact for any constant or linear function and any $\x\in T$, as 
$$
\int_T x\,d\x=\int_Ty\,d\x =\int_Tz\,d\x= 0,
$$
and 
$$
\int_T1\,d\x=\frac{8}{3}.
$$
We furthermore have that if $\x$ is uniformly distributed on $T$, then $Q$ is an unbiased integration rule, as
\begin{equation}
\begin{split}
\mathbb{E}(Q(f))=& \frac{3}{8}\int_T\frac{2}{3}\left(f(\x)+f(R_0\x)+f(R_0^2\x)+f(R_0^3\x)\right)\,d\x\\
=& \frac{1}{4}\left(\int_T\left(f(\x)+f(R_0\x)+f(R_0^2\x)+f(R_0^3\x)\right)\,d\x\right)=\int_Tf(\x)\,d\x,
\end{split}
\end{equation}
as both $T$ and the uniform distribution are invariant under $R_0$. We remark that various similar constructions are possible using any four elements of the symmetry group of $T$ that sum to zero. One may employ, for example, other subgroups that are isomorphic to that generated by $R_0$, or the diagonal matrices
\begin{equation}
R_1=\text{diag}(1,1,1);\hspace{0.025cm}R_2=\text{diag}(1,-1,-1);\hspace{0.025cm}R_3=\text{diag}(-1,-1,1);\hspace{0.025cm}R_4=\text{diag}(-1,1,-1),\hspace{0.025cm}
\end{equation}
that are not closed under multiplication.
As in the triangular case, we uniformly sample using the approach of \cite[Chapter XI, Section 2]{devroye2006nonuniform} by considering $a_1,a_2$, and $a_3$ to be i.i.d. uniform variables on $[0,1]$, and then relabelled to be increasing. In barycentric coordinates, we write 
$$
\x = a_1{\bf v}_1+(a_2-a_1){\bf v}_2+(a_3-a_2) {\bf v}_3 +(1-a_4){\bf v}_4
$$
with ${\bf v}_i$ the vertices of the reference tetrahedron.

\subsubsection{$P_2^\text{tet}$ - Order-2 rule in 3D on tetrahedral elements}

\label{subsub.p2.3d}

In order to obtain a second-order rule, we invoke the full group of proper rotations under which $T$ is invariant. \edit{Whilst more complex, the construction is based on the same arguments as the $P_2^{\text{tri}}$ rule.} This is a set of 12 rotations, generated by 
\begin{equation}
\begin{split}
R_1=&\begin{pmatrix}
0 & 1 & 0\\ 0 & 0 & 1\\ 1 & 0 & 0
\end{pmatrix}; \hspace{0.05\textwidth}R_2 = \begin{pmatrix}
-1 & 0 & 0\\ 0 & 1 & 0 \\ 0 & 0 & -1
\end{pmatrix}.
\end{split}
\end{equation}
The group of 12 elements is then given by 
\begin{equation}
G=\{I,\, R_1,\, R_1^2, \,R_2,\, R_1  R_2, \,R_2  R_1,\, R_1^2  R_2,\, R_2 R_1^2,\, R_1  R_2 R_1^2,\, R_1^2 R_2  R_1,\, R_2  R_1  R_2,\, R_2  R_1^2  R_2\}.
\end{equation}

We then propose a 13-point quadrature rule of the form 
\begin{equation}
Q_M(f)=w_1\left(\sum\limits_{R\in G}f(R\x)-12f(0)\right)+\frac{8}{3}f(0), 
\end{equation}
where 
\begin{equation}
w_1=\frac{2}{15|\x|^2}
\end{equation}
whilst $\x$ is sampled from the probability distribution 
\begin{equation}
\mathcal{P}(\x)=\frac{5}{8}|\x|^2
\end{equation}
on $T$. Similarly to the second-order rule on the triangle, this corresponds to the restriction of a radially symmetric distribution to the tetrahedra, so that a rejection-based sampler may be used. This is achieved by taking the radius to be sampled by $r=\sqrt[5]{u}$ with $u\sim U(0,1)$ with the angle taken uniformly on the sphere, then rejecting the sample if it does not lie in $T$. 
The acceptance probability of this scheme is rather low, and given by 
\begin{equation}
\frac{\int_T|\x|^2\,d\x}{\int_{B_{\sqrt{3}}(0)}|\x|^2\,d\x}=\frac{2\sqrt{3}}{27\pi}\approx 4\%. 
\end{equation}
Thus, similarly to the order-2 rule in 2D, we produce a database of 10,000 samples in a pre-training stage, which are called upon during the optimisation stage. \Cref{app.P2.3d} shows that this is an unbiased rule, whilst  \Cref{subsecUnifEstimatesP2} proves that $Q_M$ satisfies the uniform estimate
$$
\left|Q_M(f)-\int_Tf(\x)\,d\x\right|\lesssim ||\nabla^{3}f||_\infty, 
$$
so that, when $Q_M$ is extended via a uniform mesh of $N$ elements to a global quadrature rule $Q$ using i.i.d. sampling, then, for $f\in C^3(\bar{\Omega})$, we obtain the variance estimate
$$
\text{Var}(Q(f))\lesssim N^{-\frac{8}{3}}||\nabla^3f||_\infty^2.
$$

\subsection{The effect of order in smooth problems}
\label{subsecOrderFixed}

We now consider the variances of various unbiased quadrature rules with fixed, smooth integrands. The variance scaling of an order-$p$, mesh-based rule converges asymptotically as $O(N^{-1-\frac{2p+2}{d}})$ for an order $p$ mesh-based rule is only valid in the asymptotic limit as $N\to\infty$. To study the pre-asymptotic regime numerically, we consider a purely empirical study, and consider the integration of the continuum loss function for the DRM at the exact solution for our model problems, outlined in \Cref{subsec.Model.Architecture}. 

In each example, we consider i.i.d. samples of our quadrature rules using a uniform partition of the domain $[0,1]^d$. In one dimension, as all master elements are intervals, this is trivial. In two dimensions, for the $P_0$ and $P_3$ rules - based on square master elements - we divide the domain into uniform squares. For the $P_1^\text{tri}$ and $P_2^\text{tri}$ rules - based on triangular elements - we begin by taking a square, uniform partition of the domain and divide each cell in two via to the diagonal (see \Cref{subfigTrianglePartition}). 
To cover $[0,1]^3$ by tetrahedra, we take the following decomposition (see \Cref{subfigTetrahedraMesh}): First, we cover a reference cube $[0,1]^3$ by 5 tetrahedra with a standard partitioning. We define the vertices $p_1=(1,1,1)$, $p_2=(0,0,1)$, $p_3=(0,1,0)$, $p_4=(1,0,0)$, for each we define a complement $\bar{p}_i$, obtained by swapping zeroes and ones. For example, $\bar{p}_2=(1,1,0)$. We then take a central element, twice as large as the others, with vertices $p_1,p_2,p_3,p_4$. The remaining four elements, which are similar, correspond to swapping one of $p_i$ with $\bar{p}_i$. With this partition, we then break $\Omega$ into cubes, using a straightforward affine mapping to obtain the partition of $E_n$ from the reference partition. 

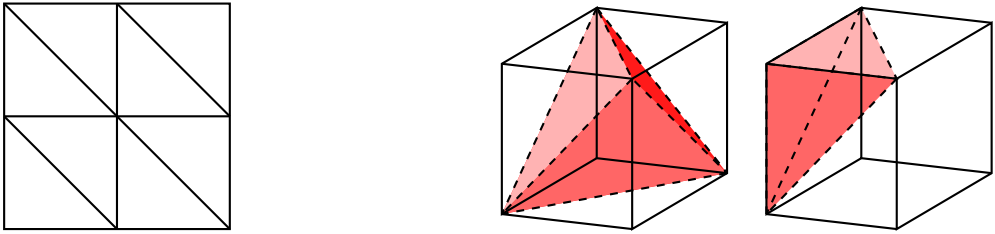
\begin{figure}[H]\begin{center}
\begin{subfigure}[t]{0.3\textwidth}\begin{center}
\begin{tikzpicture}[scale=1.5]
\begin{scope}
    \draw[thick] (0,0) -- (2,0) -- (2,2) -- (0,2) -- cycle;
    \draw[thick] (0,2) -- (2,0);
    \draw[thick] (1,0) -- (0,1);
    \draw[thick] (2,1) -- (1,2);
    \draw[thick] (1,0) -- (1,2);
    \draw[thick] (0,1) -- (2,1);
\end{scope}
\end{tikzpicture}\end{center}\caption{Illustration of the partition of $[0,1]^2$ into triangular elements.\label{subfigTrianglePartition}}\end{subfigure}
\hspace{0.05\textwidth}\begin{subfigure}[t]{0.6\textwidth}\begin{center}
\begin{tikzpicture}[scale=2]

\begin{scope}[rotate around y=165]
\coordinate (A) at (1,1,1);
\coordinate (B) at (0,0,1);
\coordinate (C) at (0,1,0);
\coordinate (D) at (1,0,0);

\fill[red!90] (A)--(B)--(D)--cycle;
\fill[red!30] (A)--(D)--(C)--cycle;
\fill[red!60] (D)--(B)--(C)--cycle;

\draw[thick,dashed,black] (A)--(B)--(C)--cycle;
\draw[thick,dashed,black] (A)--(D);
\draw[thick,dashed,black] (B)--(D);
\draw[thick,dashed,black] (C)--(D);

\draw[thick] (0,0,0)--(0,1,0)--(0,1,1)--(0,0,1)--cycle;
\draw[thick] (1,0,0)--(1,1,0)--(1,1,1)--(1,0,1)--cycle;

\draw[thick] (0,0,0)--(1,0,0);
\draw[thick] (0,1,0)--(1,1,0);
\draw[thick] (0,1,1)--(1,1,1);
\draw[thick] (0,0,1)--(1,0,1);
\end{scope}
\end{tikzpicture}\hspace{0.05\textwidth}\begin{tikzpicture}[scale=2]

\begin{scope}[rotate around y=165]
\coordinate (A) at (1,1,1);
\coordinate (B) at (1,1,0);
\coordinate (C) at (0,1,0);
\coordinate (D) at (1,0,0);

\fill[red!30] (A)--(B)--(C)--cycle;

\fill[red!60] (D)--(B)--(C)--cycle;

\draw[thick,dashed,black] (A)--(B)--(C)--cycle;
\draw[thick,dashed,black] (A)--(D);
\draw[thick,dashed,black] (B)--(D);
\draw[thick,dashed,black] (C)--(D);

\draw[thick] (0,0,0)--(0,1,0)--(0,1,1)--(0,0,1)--cycle;
\draw[thick] (1,0,0)--(1,1,0)--(1,1,1)--(1,0,1)--cycle;

\draw[thick] (0,0,0)--(1,0,0);
\draw[thick] (0,1,0)--(1,1,0);
\draw[thick] (0,1,1)--(1,1,1);
\draw[thick] (0,0,1)--(1,0,1);
\end{scope}
\end{tikzpicture}
\end{center}\caption{Illustration of the partition of $[0,1]^3$ into tetrahedral integration elements, showing the central element and one of the four similar elements.\label{subfigTetrahedraMesh}}
\end{subfigure}
\caption{The triangular and tetrahedral decompositions of $[0,1]^d$.}\end{center}
\end{figure}

\Cref{figVarianceTests} shows variance estimates of different quadrature rules. Each data point corresponds to the variance estimated by taking 1,000 loss evaluations at a given number of quadrature points. The dashed lines correspond to the asymptotic scalings $CN^{-1-\frac{2p+2}{d}}$, with $C$ obtained via interpolating over the last data point. 

\begin{figure}[H]\begin{center}
\begin{subfigure}[t]{0.45\textwidth}
\begin{center}
\includegraphics[width=0.95\textwidth]{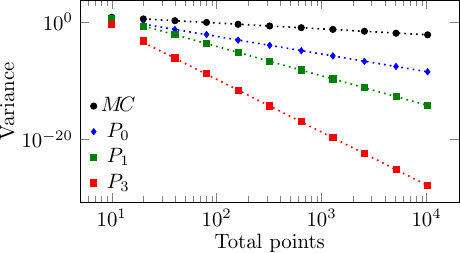}
\caption{One-dimensional example.\label{fig.1D.var.order} }
\end{center}
\end{subfigure}
\begin{subfigure}[t]{0.45\textwidth}
\begin{center}
\includegraphics[width=0.95\textwidth]{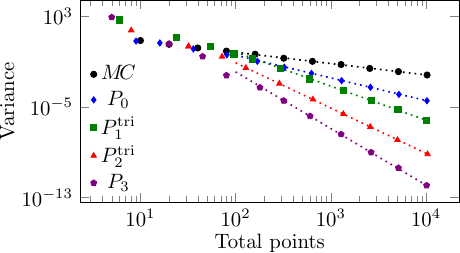}
\caption{Two-dimensional example.\label{fig.2D.var.order}}
\end{center}
\end{subfigure}
\begin{subfigure}[t]{0.45\textwidth}
\begin{center}
\includegraphics[width=0.95\textwidth]{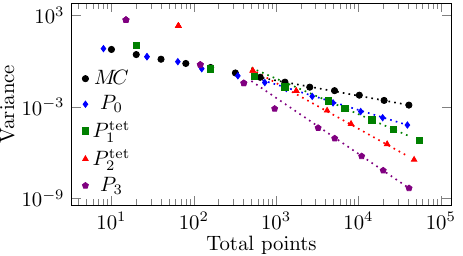}
\caption{Three-dimensional example.\label{fig.3D.var.order}}
\end{center}
\end{subfigure}
\caption{Empirically determined variances of unbiased integration rules on a fixed integrand.}\label{figVarianceTests}
\end{center}
\end{figure}

The effect of the order is most striking in the one-dimensional examples (see \Cref{fig.1D.var.order}), with the rules agreeing with their asymptotic scalings even with few integration points. The gains become significant even with a modest number of integration points, with $P_3$ outperforming $\MC$ by almost ten orders of magnitude at 96 integration points. We also highlight that $P_0$ - essentially a stratified Monte Carlo strategy - achieves smaller variance with 96 integration points than vanilla Monte Carlo quadrature with over 10,000 points. 

In two dimensions,  \Cref{fig.2D.var.order} shows little differences between the considered integration rules up to 100 integration points. In particular, the rules of order 0 to 3 all behave quantitatively like a vanilla Monte Carlo scheme. This suggests that, in a poor integration regime, mesh-based, stochastic, Gauss-type rules may not offer any significant difference with respect to MC. Beyond around 100 integration points, we see the rules following their predicted scalings, with the $P_3$ rule again outperforming the MC rule by several orders of magnitude with 10,000 integration points. Despite the gains, we note that they are less drastic than the one-dimensional case, again reflecting the curse of dimensionality. 

In the three-dimensional example of \Cref{fig.3D.var.order}, the limitations seen in two dimensions become even more apparent, requiring almost 1,000 points before the Gauss-type rules begin to behave distinctly from the Monte Carlo scheme. Nonetheless, we see that there are still significant gains to be made by using higher-order rules when employing a significant number of integration points.

\section{Numerical Results with DRM}
\label{secDRMResults}

\subsection{DRM - One-dimensional examples}

For fixed integrands, it was shown in \Cref{subsecOrderFixed} how the use of high-order rules can lead to a significant reduction in variance for a fixed number of quadrature points. As the effects are most dramatic in one dimension, we consider a more extensive study. In all cases, we use the model problem and architecture outlined in \Cref{subsec.Model.Architecture}, along with the Adam variant given in \Cref{subsecAdam}. We highlight that the NN is initialised identically for each experiment. We compare the rules $\MC,P_0,P_1$, and $P_3$ outlined in \Cref{subsec.Old.quad} in three regimes - poor, intermediate, and good integration.

\subsubsection{Case 1 - Poor integration}

First, we consider the poor integration case. We take six integration points for each of the experiments and we employ 100,000 iterations of the Adam optimiser described in \Cref{subsecAdam}. During training, we estimate the $H^1$-error at each iteration by employing a two-point deterministic Gaussian quadrature rule on a uniform mesh of 500 elements, leading to 1,000 points. The results are presented in \Cref{fig.DRM.1D_order_6}. 

\begin{figure}[H]\begin{center}
\begin{subfigure}[t]{0.45\textwidth}
\begin{center}
\includegraphics[width=0.95\textwidth]{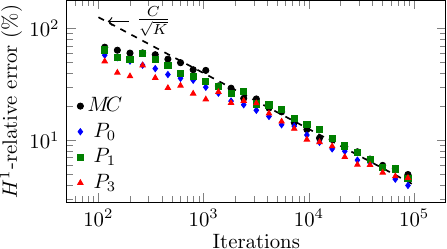}
\caption{Mean $H^1$ relative errors (\%) over logarithmically-spaced intervals. \label{fig.DRM.1D_order_6.H1}}
\end{center}
\end{subfigure}\hspace{0.05\textwidth}
\begin{subfigure}[t]{0.45\textwidth}
\begin{center}
\includegraphics[width=0.95\textwidth]{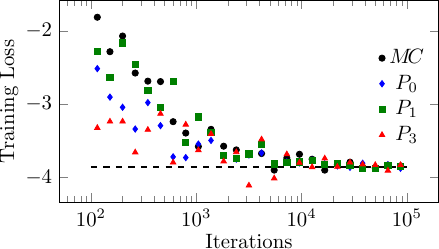}
\caption{Mean of the discretised loss over logarithmically-spaced intervals. The dashed line indicates the exact minimum of the continuum loss.}\label{fig.DRM.1D_order_6.loss}
\end{center}
\end{subfigure}
\begin{subfigure}[t]{0.45\textwidth}
\begin{center}
\includegraphics[width=0.95\textwidth]{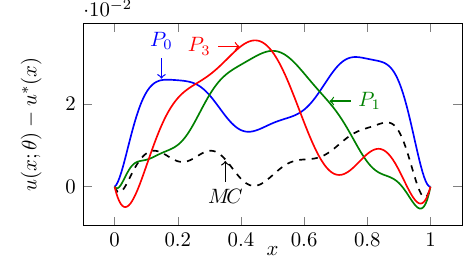}
\caption{Pointwise errors of the approximated solutions.\label{fig.DRM.1D_order_6.sols}}
\end{center}
\end{subfigure}\hspace{0.05\textwidth}
\begin{subfigure}[t]{0.45\textwidth}
\begin{center}
\includegraphics[width=0.95\textwidth]{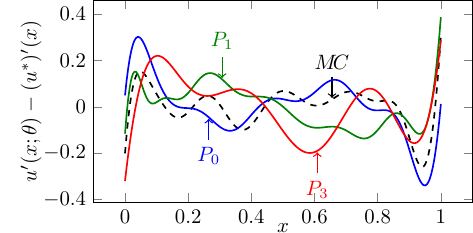}
\caption{Pointwise errors of the derivatives of the approximate solutions.\label{fig.DRM.1D_order_6.dsols}}
\end{center}
\end{subfigure}
\begin{subfigure}[t]{0.45\textwidth}
\begin{center}
\includegraphics[width=0.95\textwidth]{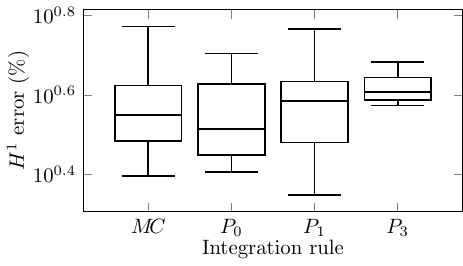}
\caption{Box plots of the relative $H^1$-error taken over the last 1,000 iterations.\label{fig.DRM.1D_order_6.boxh1}}
\end{center}
\end{subfigure}\hspace{0.05\textwidth}
\begin{subfigure}[t]{0.45\textwidth}
\begin{center}
\includegraphics[width=0.95\textwidth]{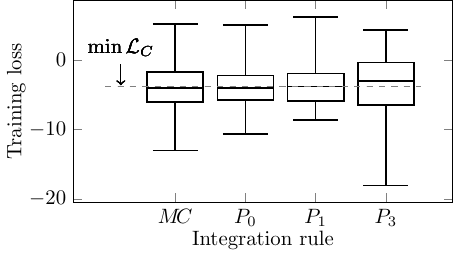}
\caption{Box plots of the training loss taken over the last 1,000 iterations.\label{fig.DRM.1D_order_6.boxloss}}
\end{center}
\end{subfigure}
\caption{1D DRM with poor integration ($N=6$).}\label{fig.DRM.1D_order_6}
\end{center}
\end{figure}

When employing poor integration, we see that all rules behave similarly, with the large variance leading to slow convergence, visible by the evolution of the $H^1$-error in \Cref{fig.DRM.1D_order_6.H1} \edit{with all converging to final values of around 3\% in all cases. Comparing to a representative curve of the form $\frac{C}{\sqrt{K}}$, we observe a strong agreement with the theoretical results, since the errors remain roughly parallel in each case.} The errors in the solutions and their derivatives, shown in \Cref{fig.DRM.1D_order_6.sols,fig.DRM.1D_order_6.dsols} are distinct, but of similar magnitude. As the results are inherently noisy, \Cref{fig.DRM.1D_order_6.boxh1,fig.DRM.1D_order_6.boxloss} show box plots of the error and training loss over the last 1,000 iterations, suggesting little qualitative difference between the rules, whilst the $P_3$ rule performs marginally worse. 

\subsubsection{Case 2 - Moderate integration}

Next, we consider the same experiment with a moderate number of integration points - 32 - for each rule. As we have less noisy integration, we employ the Adam optimiser for 25,000 iterations, whilst the rest of the experimental setup remains the same as in Case 1. 

\begin{figure}[H]\begin{center}
\begin{subfigure}[t]{0.45\textwidth}
\begin{center}
\includegraphics[width=0.95\textwidth]{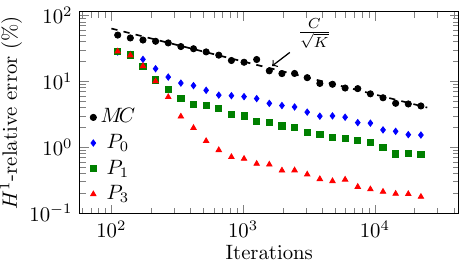}
\caption{Mean $H^1$ relative errors (\%) over logarithmically-spaced intervals. }\label{fig.DRM.1D_order_32.H1}
\end{center}
\end{subfigure}\hspace{0.05\textwidth}
\begin{subfigure}[t]{0.45\textwidth}
\begin{center}
\includegraphics[width=0.95\textwidth]{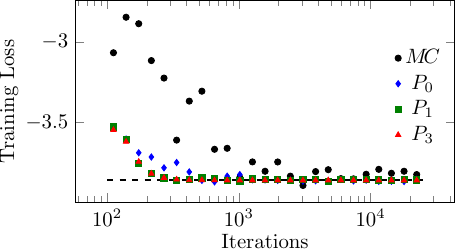}
\caption{Mean of the discretised loss over logarithmically-spaced intervals. The dashed line indicates the exact minimum of the continuum loss.}\label{fig.DRM.1D_order_32.loss}
\end{center}
\end{subfigure}
\begin{subfigure}[t]{0.45\textwidth}
\begin{center}
\includegraphics[width=0.95\textwidth]{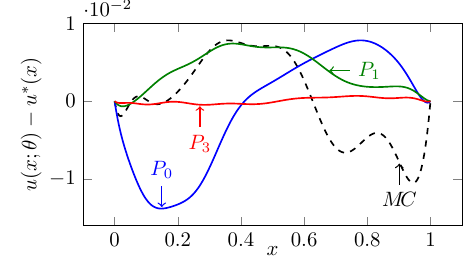}
\caption{Pointwise errors of the approximated solutions.}\label{fig.DRM.1D_order_32.sols}
\end{center}
\end{subfigure}\hspace{0.05\textwidth}
\begin{subfigure}[t]{0.45\textwidth}
\begin{center}
\includegraphics[width=0.95\textwidth]{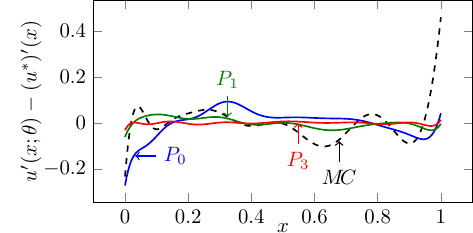}
\caption{Pointwise errors of the derivatives of the approximate solutions.}\label{fig.DRM.1D_order_32.dsols}
\end{center}
\end{subfigure}
\begin{subfigure}[t]{0.45\textwidth}
\begin{center}
\includegraphics[width=0.95\textwidth]{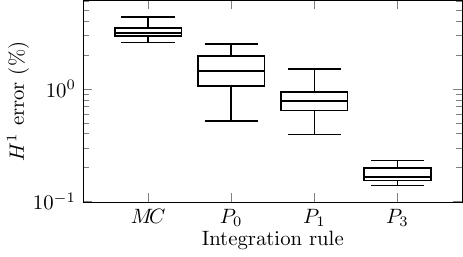}
\caption{Box plots of the relative $H^1$-error taken over the last 1,000 iterations.}\label{fig.DRM.1D_order_32.boxh1}
\end{center}
\end{subfigure}\hspace{0.05\textwidth}
\begin{subfigure}[t]{0.45\textwidth}
\begin{center}
\includegraphics[width=0.95\textwidth]{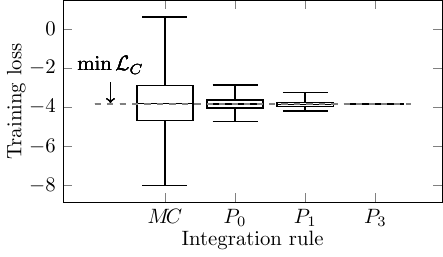}
\caption{Box plots of the training loss taken over the last 1,000 iterations.}\label{fig.DRM.1D_order_32.boxloss}
\end{center}
\end{subfigure}
\caption{1D DRM with moderate integration ($N=32$).}\label{fig.DRM.1D_order_32}
\end{center}
\end{figure}

We present the results in \Cref{fig.DRM.1D_order_32}. In this regime, we see a more striking difference, depending on the quadrature rule employed. In \Cref{fig.DRM.1D_order_32.H1}, we see that each relative $H^1$-error appears to follow the theoretical $\frac{C}{\sqrt{K}}$ scaling in each case. Nonetheless, the significant differences in the variance lead to more favourable constants for the higher-order rules. The faster convergence is also observable to a lesser extent when the training loss is viewed in \Cref{fig.DRM.1D_order_32.loss}. The reduction of error in the solutions and their derivatives is visible in \Cref{fig.DRM.1D_order_32.sols,fig.DRM.1D_order_32.sols}. In the box plots of the relative $H^1$-error sampled over the last 1,000 iterations in \Cref{fig.DRM.1D_order_32.boxh1}, the monotonic improvement with respect to the order is clearly visible. \Cref{fig.DRM.1D_order_32.boxloss} shows how the variation of the discretised loss is significantly reduced by using high-order rules.

\subsubsection{Case 3 - Good integration}

We now turn to the case of good integration, employing 252 integration points in each case. As we expect lower variances, we train for fewer iterations, using 10,000 with the Adam optimiser. Other than these details, the rest of the experiment is identical to those of Cases 1 and 2. The numerical results are shown in \Cref{fig.DRM.1D_order_252}. 

\begin{figure}[H]\begin{center}
\begin{subfigure}[t]{0.45\textwidth}
\begin{center}
\includegraphics[width=0.95\textwidth]{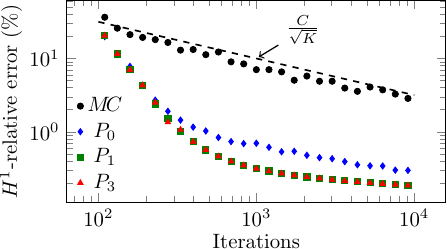}
\caption{Mean $H^1$ relative errors (\%) over logarithmically-spaced intervals. \label{fig.DRM.1D_order_252.H1}}
\end{center}
\end{subfigure}\hspace{0.05\textwidth}
\begin{subfigure}[t]{0.45\textwidth}
\begin{center}
\includegraphics[width=0.95\textwidth]{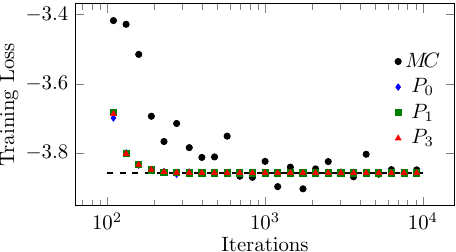}
\caption{Mean of the discretised loss over logarithmically-spaced intervals. The dashed line indicates the exact minimum of the continuum loss.\label{fig.DRM.1D_order_252.loss}}
\end{center}
\end{subfigure}
\begin{subfigure}[t]{0.45\textwidth}
\begin{center}
\includegraphics[width=0.95\textwidth]{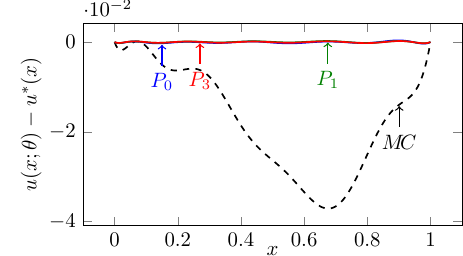}
\caption{Pointwise errors of the approximated solutions.\label{fig.DRM.1D_order_252.sols}}
\end{center}
\end{subfigure}\hspace{0.05\textwidth}
\begin{subfigure}[t]{0.45\textwidth}
\begin{center}
\includegraphics[width=0.95\textwidth]{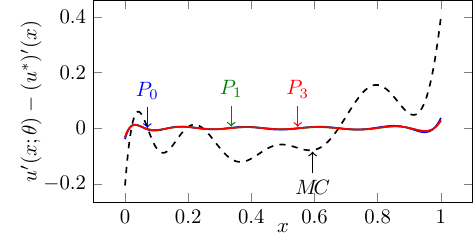}
\caption{Pointwise errors of the derivatives of the approximate solutions.\label{fig.DRM.1D_order_252.dsols}}
\end{center}
\end{subfigure}
\begin{subfigure}[t]{0.45\textwidth}
\begin{center}
\includegraphics[width=0.95\textwidth]{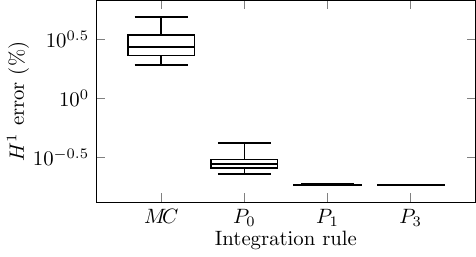}
\caption{Box plots of the relative $H^1$-error taken over the last 1,000 iterations.\label{fig.DRM.1D_order_252.boxh1}}
\end{center}
\end{subfigure}\hspace{0.05\textwidth}
\begin{subfigure}[t]{0.45\textwidth}
\begin{center}
\includegraphics[width=0.95\textwidth]{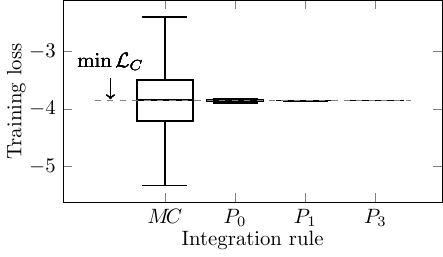}
\caption{Box plots of the training loss taken over the last 1,000 iterations.\label{fig.DRM.1D_order_252.boxloss}}
\end{center}
\end{subfigure}
\caption{1D DRM with good integration ($N=252$). }\label{fig.DRM.1D_order_252}
\end{center}
\end{figure}

The $H^1$-errors, when employing Gauss-type rules, no longer follow the $\frac{C}{\sqrt{K}}$ scaling, \edit{as they fail to remain parallel to a representative curve of this form}. This is indicative of the limitations of the NN to approximate the exact solution and the potential presence of local minimisers of the discretised loss. In the best-case scenario, the iterates $\theta_k$ converge to a critical point $\theta^*$ of the discretised loss at a rate of $\frac{C}{\sqrt{K}}$, however, the realisation of the NN at $\theta^*$ is not expected to be the exact solution of the continuum problem. Thus, we no longer see a power-law structure in the log-log plot once we are near an optimal value of $\theta$. The near-indistinguishability of the behaviour of the $P_1$ and $P_3$ rules suggests that integration errors are insignificant during training. Conversely, we see that the $\MC$ rule admits errors an order of magnitude higher at the end of training with significant variation in its loss at the end of training.

\subsection{Two-dimensional study}\label{subsec_2D_DRM}

\edit{Owing to the curse of dimensionality in the variance, represented via the denominator $d$ in \eqref{eqVarEst}, the effect of order in higher dimensional problems is diminished. As such, in the two-dimensional examples here and the three-dimensional examples in the sequel, rather than comparing all available quadrature rules, we will consider more focused experiments with quadrature schemes designed to obtain a qualitative comparison between rules. The code containing all quadrature rules outlined in \Cref{subsec.Old.quad,subsec.New.Quad} can be accessed via the repository in \Cref{secDataAvailability} for independent experimentation.}
 
We train an NN with the DRM loss and two different integration rules: a vanilla Monte Carlo scheme ($\MC$) and the order-2 rule on triangular elements ($P_2^\text{tri}$), as described in \Cref{subsec.New.Quad}. For each, we consider two integration regimes: a poor integration regime consisting of 128 points and a fine regime consisting of 2312 points. We empirically selected these numbers so that variances of the discrete (stochastic) losses evaluated at the exact solution are approximately equal for the $\MC$ quadrature rule in the fine regime and the $P^\text{tet}_2$ rule in the poor regime (see \Cref{tableVars2DFew}). Thus, we expect the fine $\MC$ quadrature and poor $P^\text{tri}_2$ quadrature to yield similar results during training; however, the $P_2^\text{tri}$ rule employs a far lower number of quadrature points and, correspondingly, a lower computational cost. 

We consider the 2D model problem and architecture  described in \Cref{subsec.Model.Architecture} with 20,000 iterations of the Adam optimiser outlined in \Cref{subsecAdam}.

\begin{figure}[H]\begin{center}
\begin{subfigure}[t]{0.45\textwidth}
\begin{center}
\includegraphics[width=0.95\textwidth]{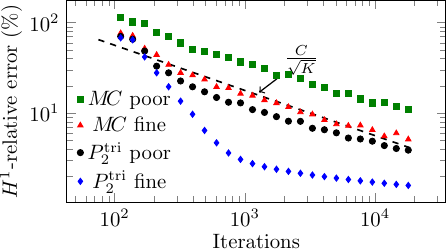}
\caption{The mean $H^1$ relative errors (\%) over logarithmically-spaced intervals. }\label{fig.DRM.2D_order_few.H1}
\end{center}
\end{subfigure}\hspace{0.05\textwidth}
\begin{subfigure}[t]{0.45\textwidth}
\begin{center}
\includegraphics[width=0.95\textwidth]{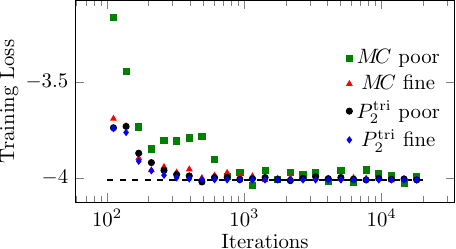}
\caption{The mean of the discretised loss over logarithmically-spaced intervals. The dashed line indicates the exact minimum of the continuum loss.}\label{fig.DRM.2D_order_few.loss}
\end{center}
\end{subfigure}
\begin{subfigure}[t]{0.45\textwidth}
\begin{center}
\includegraphics[width=0.95\textwidth]{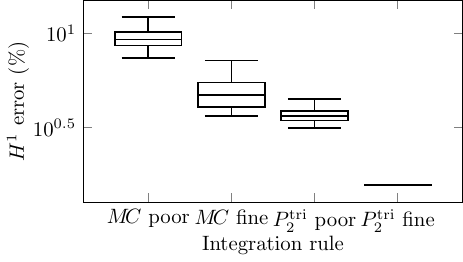}
\caption{Box plots of the relative $H^1$-error taken over the last 1,000 iterations.}\label{fig.DRM.2D_order_few.boxh1}
\end{center}
\end{subfigure}\hspace{0.05\textwidth}
\begin{subfigure}[t]{0.45\textwidth}
\begin{center}
\includegraphics[width=0.95\textwidth]{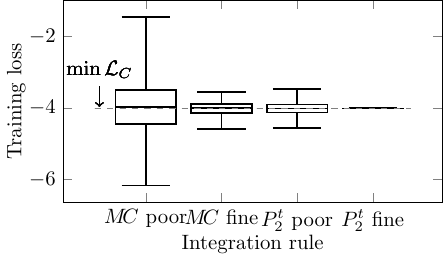}
\caption{Box plots of the training loss taken over the last 1,000 iterations.}\label{fig.DRM.2D_order_few.boxloss}
\end{center}
\end{subfigure}
\caption{2D DRM with different integration schemes. }\label{fig.DRM.2D_order_few}
\end{center}
\end{figure}

The numerical results are presented in \Cref{fig.DRM.2D_order_few}. We observe that the poor $P_2^\text{tri}$ and fine $\MC$ schemes produce generally similar results, both in terms of the $H^1$-error and training losses, whilst the $P_2^\text{tri}$ rule utilises around 20 times fewer integration points. In both the poor and fine integration regimes, the errors when the $P_2^\text{tri}$ rule is employed is around half of those when the $\MC$ rule with the same number of quadrature points is employed. Furthermore, in both regimes, in around 2,000 iterations, the $P_2^\text{tri}$ rule has reached similar $H^1$-errors as the $\MC$ rule achieves after 20,000. \edit{We see that both $\MC$ rules and poor $P_2^{\text{tri}}$ remain parallel to a representative curve of the form $\frac{C}{\sqrt{K}}$, whilst the fine $P_2^{\text{tri}}$ rule appears to be reaching a plateau, indicative of convergence to a critical point that does not correspond to the exact solution.} Furthermore, we include in \Cref{tableVars2DFew} the variances of the discretised loss using each scheme, both at the exact solution and the corresponding approximate solutions obtained at the end of training. These are estimated by taking 1,000 samples in each case. We see that the variances of $\MC$ in the fine regime and $P_2^\text{tri}$ in the poor regime are similar, both at the exact and their approximate solutions, with significant gains appearing when the $P_2^\text{tri}$ rule is used with the same number of integration points. Whilst the variances of the loss using the fine $\MC$ and poor $P_2^\text{tri}$ rules are similar, we still observe a slight, but consistent, improvement using the $P_2^\text{tri}$ rule. As the variances of the losses differ by less than $10\%$, we study this further by estimating the stochastic gradient at the trained solutions, recalling that the asymptotic behaviour is dictated by the properties of the covariance matrix of the loss rather than the loss itself. To quantify the size of this high-dimensional object, we employ its trace as a measure of total variance and its largest eigenvalue as a measure of the greatest component of the variance. Despite the similarity in the variance of the loss, the trace of the covariance matrix is estimated to be roughly four times greater in the fine $\MC$ rule than in the poor $P_2^\text{tri}$ rule, whilst the largest eigenvalue is almost 10 times larger. It is this discrepancy that appears to lead to the difference in solution quality using the two methods. We also highlight that, in all cases, the trace - corresponding to the sum of over 2,000 eigenvalues - is of comparable magnitude to the largest eigenvalue of the covariance matrix. In the fine $\MC$ case, around two thirds of the trace arises from a single eigenvalue. This is suggestive of the potential of focusing integration strategies for variance reduction on a small number of highly influential components of the integration error, even when the number of trainable parameters is large.

\begin{table}
\begin{center}
\begin{tabular}{|c||c|c|c|c|}
\hline
& Loss (exact) & Loss (Approximate) & $\text{Tr}(\text{Cov})$ & $\lambda_{\max}(\text{Cov})$ \\
\hline
$\MC$ poor & $0.552$ & $0.551$ &$30.7$ & $16.4$\\
\hline
$\MC$ fine & $0.0295$ & $ 0.0294$ & $1.97$ & $1.23$\\
\hline
$P^\text{tri}_2$ poor &$0.0280$ & $0.0274$ & $0.504$& $0.130$ \\
\hline
$P^\text{tri}_2$ fine & $2.26\times 10^{-7}$ & $2.17\times 10^{-7}$ & $5.73\times 10^{-6}$ & $1.43\times 10^{-6}$\\
\hline
\end{tabular}
\caption{Estimated variances of the losses and covariances of the stochastic gradient when evaluated at the exact solution and the corresponding approximate solution obtained via the DRM in 2D.}
\label{tableVars2DFew}\end{center}
\end{table}

\subsection{Three-dimensional study}\label{subsec_3D_DRM}

We now compare three $\MC$ schemes with the order-2 rule for tetrahedral elements $(P^\text{tet}_2)$. For the $P^\text{tet}_2$ scheme, we employ 1,755 points, corresponding to a $3\times 3\times 3$ partition of the unit cube into 27 uniform cubes, each filled with 5 tetrahedra. We choose three $\MC$ schemes: The ``poor" regime corresponds to the same number (1,755) of quadrature points. The ``mid" regime is chosen, as in the two-dimensional study, so that the variance of the loss at the exact solution is similar to that of the employed $P_2^\text{tet}$ rule (see \Cref{table_3D_vars_etc}), consisting of 5,320 points. The number of points in the ``fine" regime, 31,000, is empirically chosen so that the variance of the stochastic derivative is of a similar magnitude to that of the $P^\text{tet}_2$ rule. As in \Cref{tableVars2DFew}, this is done by estimating the covariance matrix over 500 realisations, and we use its trace as the metric we aim to match. 

\begin{figure}[H]\begin{center}
\begin{subfigure}[t]{0.45\textwidth}
\begin{center}
\includegraphics[width=0.95\textwidth]{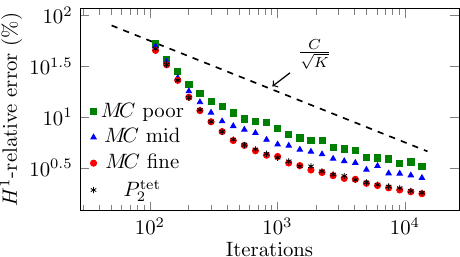}
\caption{The mean $H^1$ relative errors (\%) over logarithmically-spaced intervals. \label{fig.DRM.3D_order_few.H1}}
\end{center}
\end{subfigure}\hspace{0.05\textwidth}
\begin{subfigure}[t]{0.45\textwidth}
\begin{center}
\includegraphics[width=0.95\textwidth]{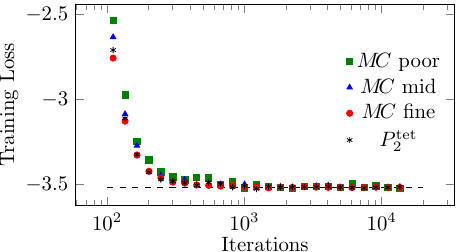}
\caption{The mean of the discretised loss over logarithmically-spaced intervals. The dashed line indicates the exact minimum of the continuum loss.}\label{fig.DRM.3D_order_few.loss}
\end{center}
\end{subfigure}
\begin{subfigure}[t]{0.45\textwidth}
\begin{center}
\includegraphics[width=0.95\textwidth]{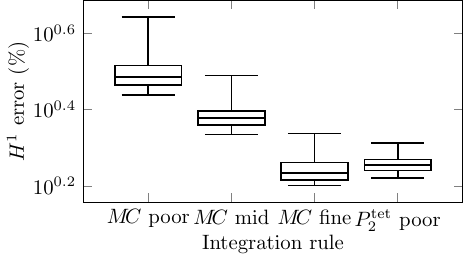}
\caption{Box plots of the relative $H^1$-error taken over the last 1,000 iterations.}\label{fig.DRM.3D_order_few.boxh1}
\end{center}
\end{subfigure}\hspace{0.05\textwidth}
\begin{subfigure}[t]{0.45\textwidth}
\begin{center}
\includegraphics[width=0.95\textwidth]{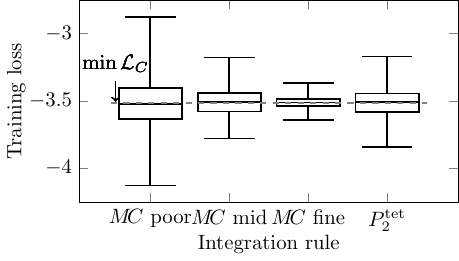}
\caption{Box plots of the training loss taken over the last 1,000 iterations.}\label{fig.DRM.3D_order_few.boxloss}
\end{center}
\end{subfigure}
\caption{3D DRM with different integration schemes. }\label{fig.DRM.3D_order_few}
\end{center}
\end{figure}

The results are presented in \Cref{fig.DRM.3D_order_few}. We see that the evolution of the $H^1$-error in the fine $\MC$ rule and $P^\text{tet}_2$ rule are nearly indistinguishable. We see, however, that the mid $\MC$ rule exhibits consistently poorer results during training. These results highlight clearly how the covariance of the stochastic derivative is a limiting factor of convergence, and how an integration strategy designed on minimising the variance of the stochastic loss function need not be equally applicable to variance reduction in the stochastic derivative - whilst the mid $\MC$ and $P_2^{\text{tet}}$ rules admit similar variances when evaluating the loss, the measures of the covariance of the stochastic derivative are highly disparate. 

\begin{table}[H]\begin{center}
\begin{tabular}{|c||c|c|c|c|c|}
\hline
& Points & Loss (exact) & Loss (approximate) & $\text{Tr}(\text{Cov})$ &  $\lambda_{\max}(\text{Cov})$\\
\hline
$\MC$ poor & $\bf 1,755 $&  $0.0340$ &$ 0.0333$ & $3.81$ & $2.58$ \\ 
\hline
$\MC$ mid & $5,320$ & $\bf 0.0105$ & {$\bf  0.0105$} &  $1.39$  & $0.964$ \\
\hline
$\MC$ fine & $31,000$& $0.00156$ & $0.00154$ & $ \bf 0.297$ & $0.214$ \\
\hline
$P_2^{\text{tet}}$ & $\bf 1,755 $ & $\bf 0.0110$ & {$\bf  0.0106$} & {$ \bf 0.268$} & $0.140$\\
\hline
\end{tabular}
\caption{Estimations of the variances of the loss at the exact and corresponding approximate solutions and metrics of the covariances of the stochastic gradient at the corresponding approximate solutions. Each number is estimated via 500 evaluations. Comparable metrics between different rules are highlighted in bold.\label{table_3D_vars_etc}}\end{center}
\end{table}

\section{Conclusions}\label{secConclusions}

We have studied various aspects of integration when using the Deep Ritz Method to solve PDEs using Neural Networks, focusing on low-dimensional problems. The Deep Ritz Method, based on the weak-formulation of the PDE, invokes a loss function lacking the interpolation property, making it particularly susceptible to integration errors. We have demonstrated how deterministic integration can lead to disastrous overfitting. Whilst Monte Carlo is taken as the standard integration method when using NNs, we show how the use of stochastic, mesh-based, unbiased and high-order integration rules in low-dimensional, smooth problems can lead to significant gains in convergence at an equivalent computational cost. Whilst such rules are expected to be significantly more accurate when a large number of integration points are used, we empirically show that when few integration points are used, they behave at worst like a vanilla Monte Carlo method. When employing stochastic integration rules, we have shown how the use of biased quadrature rules may also lead to erroneous results, and theoretical results suggest that even an arbitrarily fine integration mesh cannot compensate for bias. 

We have proposed new order-1 and order-2 unbiased integration rules on triangular elements in 2D and tetrahedral elements in 3D, which may prove useful in more complex geometries due to their amenability for meshing. We have further highlighted how distinct integration strategies that produce similar errors when evaluating the loss function may lead to highly disparate errors when evaluating the stochastic derivative and deliver significantly different results under the same training scheme. In particular, this emphasises that integration schemes should be designed to properly integrate the stochastic derivative that, as a high-dimensional object, proves challenging.

This work has inherent limitations that restrict its applicability to more general problems. First, we predominantly consider only the integration of the loss itself, whilst we see in the examples of \Cref{subsec_2D_DRM,subsec_3D_DRM} that it is the integration of the stochastic gradient that is the key object. This distinction is expected to be especially significant when solutions are singular or highly oscillatory, leading to more extreme differences in behaviour between the loss and its gradient. As such, further study and design of appropriate unbiased quadrature rules for singular integrands as well as methods for variance reduction of the high-dimensional stochastic derivative are necessary avenues for future work. Related to this, we have only considered prescribed, uniform integration meshes. In the style of Finite Element Methods, one may consider refinement strategies based on the principles of $h$- (element size), $p$- (order) and $r$- (node location) refinement. When localised features are present, well-implemented refinement strategies should lead to significant variance reduction and, thus, faster convergence. 

\section{Acknowledgements}
This work has received funding from the following Research Projects/Grants: 
European Union’s Horizon Europe research and innovation programme under the Marie Sklodowska-Curie Action MSCA-DN-101119556 (IN-DEEP). 
PID2023-146678OB-I00 funded by MICIU/AEI /10.13039/501100011033 and by FEDER, EU;
BCAM Severo Ochoa accreditation of excellence CEX2021-001142-S funded by MICIU/AEI/10.13039/501100011033; 
Basque Government through the BERC 2022-2025 program;
RUL-ET(KK-2024/00086), funded by the Basque Government through ELKARTEK;
Consolidated Research Group MATHMODE (IT1866-26) of the UPV/EHU given by the Department of Education of the Basque Government; 
BCAM-IKUR-UPV/EHU, funded by the Basque Government IKUR Strategy and by the European Union NextGenerationEU/PRTR.

\paragraph{Declaration of generative AI and AI-assisted technologies in the writing process}
During the preparation of this work the authors used Gemini in order to assist with proofreading of the manuscript and ChatGPT-4o in order to assist with the production of the corresponding code. After using this tool/service, the authors reviewed and edited the content as needed and take full responsibility for the content of the publication.

\section{Code Availability Statement}\label{secDataAvailability}
\edit{The code used to generate the findings of this study is openly available in \url{https://doi.org/10.5281/zenodo.19202832}.}

\bibliography{bib.bib} 

\appendix

\section{Construction of triangular and tetrahedral rules}\label{appTriTet}

\subsection{$P^2_t$ in 2D}\label{app.P2.2d}

In the spirit of \cite{siegel1985unbiased}, our aim is to exploit the symmetry of the triangle $T$ in order to produce an unbiased rule. The rule we consider takes samples $\x\in T$ according to the probability distribution $\mathcal{P}$ on $T$, and then selects a weight $w_1$ and quadrature rule as
\begin{equation}\label{eq.pwQ.p2.2d}
\begin{split}
\mathcal{P}(\x)=&\frac{1}{3w_1(\x)}=\frac{16|\x|^2}{3\sqrt{3}},\\
w_1(\x)=&\frac{\sqrt{3}}{16|\x|^2},\\
Q(f)=&w_1\left(f(\x)+f(R\x)+f(R^2\x)-3f(0)\right)+\frac{3\sqrt{3}}{4}f(0).
\end{split}
\end{equation}
We recall that $R$ is taken to be the rotation about the origin of $\frac{2\pi}{3}$, under which $T$ is invariant. For the constant function $f(\x)=1$, the term multiplied by $w_1$ vanishes, leaving only $\frac{3\sqrt{3}}{4}f(0)=|T|$, yielding exactness. For linear functions, $f(0)=0$, whilst $f(\x)+f(R\x)+f(R^2\x)=0$, so that $Q(f)=0$, again yielding the correct value. We turn to showing exactness for quadratics. 

First, we remark that $\mathcal{P}$ and $w_1$ are invariant under $R$, so that $\mathcal{P}(R\x)=\mathcal{P}(\x)$ and $w_1(\x)=w_1(R\x)$. As such, if $Q(f)$ is exact for some function $f$, it is immediately exact for $f(R\x)$. In particular, with the choice of $f(\x)=x^2$, where $\x=(x,y)$, we have that
\begin{equation}
\begin{split}
f(\x)=&x^2\\
f(R\x)=&\frac{x^2}{4}+\frac{\sqrt{3}}{2}xy+\frac{3}{4}y^2,\\
f(R^2\x)=&\frac{x^2}{4}-\frac{\sqrt{3}}{2}xy+\frac{3}{4}y^2. 
\end{split}
\end{equation}
Thus, we note that the span of $f(\x),f(R\x)$ and $f(R^2\x)$ is the same as the span of $x^2,xy$, and $y^2$. In particular, to show exactness for quadratics, it suffices to show that the integration rule is exact for $f(\x)=x^2$. Furthermore, 
$$
f(\x)+f(R\x)+f(R^2\x)=\frac{3}{2}(x^2+y^2)=\frac{3}{2}|\x|^2.
$$
As $f(0)=0$, we have 
$$
Q(f)=\frac{3w_1}{2}|\x|^2. 
$$
With this in mind, we take $w_1$ so that 
$$
\frac{3w_1}{2}|\x|^2=Q(f)=\int_Tx^2\,dx = \frac{3\sqrt{3}}{32},
$$
yielding the value in \eqref{eq.pwQ.p2.2d}. 
Finally, with the choice of $\mathcal{P}$ in \eqref{eq.pwQ.p2.2d}, we see that for any continuous function $f$, 
\begin{equation}
\begin{split}
\mathbb{E}(f)=&\int_TQ(f)\mathcal{P}(\x)\,d\x\\
=& \int_T\left(w_1(\x)\left(f(\x)+f(R\x)+f(R^2\x)-3f(0)\right)+\frac{3\sqrt{3}}{4}f(0)\right)\frac{1}{3w_1(\x)}\,d\x\\
=& \frac{1}{3}\int_T\left(f(\x)+f(R\x)+f(R^2\x)\right)\,d\x-\int_Tf(0)\,dx+\frac{3\sqrt{3}}{4}f(0)\int_T\mathcal{P}(\x)\,d\x\\
=&\int_Tf(\x)\,d\x-\frac{3\sqrt{3}}{4}f(0)+\frac{3\sqrt{3}}{4}f(0)=\int_Tf(\x)\,d\x.
\end{split}
\end{equation}

\subsection{$P^2_t$ in 3D}\label{app.P2.3d}

The construction of the $P^2_t$ rule in 3D is heuristically similar to that of 2D. We take $\x$ from the probability distribution $\mathcal{P}$ on the reference tetrahedron $T$, with the weight $w_1$ and quadrature rule $Q$ defined as

\begin{equation}
\begin{split}
\mathcal{P}(\x)=&\frac{5}{8}|\x|^2,\\
w_1=&\frac{2}{15|\x|^2},\\
Q(f)=&w_1\left(\sum\limits_{R\in G}f(R\x)-12f(0)\right)+\frac{8}{3}f(0),
\end{split}
\end{equation}
where $G$ is the order-12 group of all proper rotations under which the tetrahedron is invariant, as described in \Cref{subsub.p2.3d}. As $|T|=\frac{8}{3}$, $\int_T\x\,d\x=0$, and 
$$
\sum\limits_{R\in G} R=0, 
$$
it is immediate that the quadrature rule is exact for constant and linear functions. It remains to show that it is exact for homogeneous polynomials of order 2. 

Similarly to the 2D case, we will exploit the symmetry of the tetrahedron and the second order polynomials. If $f$ is a polynomial satisfying $f(x,y,z)=-f(\sigma_1x,\sigma_2y,\sigma_3z)$ where two of the $\sigma_i$ are $-1$ and one is $1$, then $Q(f)=0$, which is the exact value of the integral of $p$ over $T$. This holds for the quadratics $xy,yz$, and $xz$. Thus, to show exactness for quadratics, it suffices to show that the rule is exact for $x^2,y^2$ and $z^2$.

Via the symmetry property, taking $f(\x)=x^2$, we have that $f(R_1\x)=y^2$ and $f(R_1^2\x)=z^2$, where $$R_1=\begin{pmatrix}
0 & 1 & 0\\ 0 & 0 & 1\\ 1 & 0 & 0
\end{pmatrix}\in G.$$
Thus, in order to integrate all quadratic functions exactly, it is sufficient to choose $w_1$ so that $Q(f)$ is exact for $f(\x)=x^2$. Direct computation yields that 
\begin{equation}
\begin{split}
\int_Tx^2\,dx=&\frac{8}{15},\\
Q(f)=& 4w_1|\x|^2. 
\end{split}
\end{equation}
Thus, by taking $w_1(\x)=\frac{2}{15|\x|^2}$, we obtain an integration rule that is exact on $T$ for all quadratic functions. To see that $Q$ is unbiased, we observe that 
\begin{equation}
\begin{split}
\mathbb{E}(Q(f))=&\int_T\left(w_1(\x)\left(\sum\limits_{R\in G}f(R\x)-12f(0)\right)+\frac{8}{3}f(0)\right)\mathcal{P}(\x)\,d\x\\
=&\int_T\frac{1}{12}\sum\limits_{R\in G}f(R\x)\,d\x+f(0)\left(-\int_T1\,d\x+\frac{8}{3}\int_T\mathcal{P}(\x)\,d\x\right)\\
=& \frac{1}{12}\sum\limits_{R\in G}\int_Tf(R\x)\,d\x+f(0)\left(-\frac{8}{3}+\frac{8}{3}\right)\\
=& \frac{1}{12}\sum\limits_{R\in G}\int_T f(\x)\,d\x=\int_Tf(\x)\,d\x.
\end{split}
\end{equation}

\subsection{Uniform error estimates for $P^\text{tet}_2$ and $P^\text{tri}_2$}
\label{subsecUnifEstimatesP2}

\begin{proposition}
Let $T$ be the reference triangle in 2D and reference tetrahedron in 3D. For the rules $P_2^{\text{tri}}$ in 2D and $P_2^\text{tet}$ in 3D, there exists a constant $C>0$ such that, for all $f\in C^3(\bar{T})$, 
$$
\left|Q(f)-\int_Tf(\x)\,d\x\right|\leq C||\nabla^3f||_{\infty}. 
$$
\end{proposition}
\begin{proof}
We begin by employing an exact Taylor expansion of $f$ about zero, giving that 
\begin{equation}
f(\x)=f(0)+\x\cdot\nabla f(0)+\frac{1}{2}\nabla^2 f(0)\x\cdot \x +\frac{1}{6}\nabla^3 f(\eta_{\x})\cdot (\x)^{\otimes 3},
\end{equation}
where $\eta_{\x}\in T$ and $(\x)^{\otimes 3}$ denotes the third-order tensor with components $((\x)^{\otimes 3})_{ijk}=x_ix_jx_k$. The first three terms are integrated exactly by an order-two rule, so that the integration error may be expressed as 
\begin{equation}
\begin{split}
\left|Q(f)-\int_T f(\x)\,d\x\right|=& \left|Q\left(\frac{1}{6}\nabla^3 f(\eta_{\x})\cdot (\x)^{\otimes 3}\right)-\int_T\frac{1}{6}\nabla^3f(\eta_\x)\cdot(\x)^{\otimes 3}\,d\x\right|\\
\leq & \left|Q\left(\frac{1}{6}\nabla^3 f(\eta_{\x})\cdot (\x)^{\otimes 3}\right)\right|+\left|\int_T\frac{1}{6}\nabla^3f(\eta_\x)\cdot(\x)^{\otimes 3}\,d\x\right|.
\end{split}
\end{equation}
We recall that in both 2D and 3D, the quadrature weight for the non-zero points may be expressed as $\frac{c_2}{|\x|^2}$ for a constant $c_2$, whilst $\frac{1}{6}\nabla^3 f(\eta_{\x})\cdot (\x)^{\otimes 3}=0$ when $\x=0$. We may then estimate
\begin{equation}
\begin{split}
\left|Q\left(\frac{1}{6}\nabla^3 f(\eta_{\x})\cdot (\x)^{\otimes 3}\right)\right|=&\frac{1}{6}\left|\frac{c_2}{|\x|^2}\sum\limits_{R\in G}\nabla^3 f(\eta_{R\x})\cdot (R\x)^{\otimes 3}\right|\\
\leq & \frac{c_2|\x|}{6}\sum\limits_{R\in G}|\nabla^3f(\eta_{R\x})|\\
\leq & C_1||\nabla^3f||_\infty.
\end{split}
\end{equation}
The second term follows similarly, 
\begin{equation}
\begin{split}
\left|\int_T\frac{1}{6}\nabla^3f(\eta_\x)\cdot(\x)^{\otimes 3}\,d\x\right|\leq & \frac{1}{6}\int_T|\nabla^3 f(\eta_\x)|\,|\x|^3\,d\x\\
\leq & C_2||\nabla^3 f||_\infty,
\end{split}
\end{equation}
yielding the final result with $C=C_1+C_2$. 
\end{proof}

\section{Proofs of error estimates}\label{appErrorEstimates}

This appendix proves various integration error estimates obtained with the Gauss-type rules we are employing. 

\begin{proposition}
Let $\Omega_0$ be a bounded domain and $Q:C(\overline{\Omega}_0)\to\mathbb{R}$ be an order $p$ quadrature rule, 
$$
Q(f)=\sum\limits_{j=1}^J f(\x_j)w_j,
$$
with $\x_j\in\Omega_0$ and $w_j\geq 0$ for $j=1,...,J$. Then,
$$
\left|Q(f)-\int_{\Omega_0}f(\x)\,d\x\right|\lesssim ||\nabla^{p+1}f||_\infty 
$$
for all $f\in C^{p+1}(\overline{\Omega}_0)$.
\end{proposition}

\begin{proof}
Without loss of generality, we assume that $0\in\Omega_0$. We employ a Taylor polynomial approximation of $f$ with exact remainder as 
\begin{equation}
\begin{split}
f(\x)=&\sum\limits_{k=1}^{p}\frac{1}{k!}\nabla^k f(0)\cdot (\x)^{\otimes k}+\frac{1}{(p+1)!}\nabla^{p+1}f(\eta_\x)\cdot(\x)^{\otimes p+1}. 
\end{split}
\end{equation}
The notation $(\x)^{\otimes k}$ denotes the order $k$ tensor corresponding to the tensor product of $\x$ with itself $k$ times. As the sum from $k=1$ to $p$ corresponds to an order $p$ polynomial, we have that the integration error satisfies 
\begin{equation}
\begin{split}
\left|Q(f)-\int_{\Omega_0}f(\x)\,d\x\right|=&\left|\sum\limits_{j=1}^J\frac{w_j}{(p+1)!}\nabla^{p+1}f(\eta_{\x_j})\cdot(\x_j)^{\otimes p+1}-\int_{\Omega_0}\frac{1}{(p+1)!}\nabla^{p+1}f(\eta_\x)\cdot(\x)^{\otimes p+1}\,d\x\right|\\
\leq & \left|\sum\limits_{j=1}^J\frac{w_j}{(p+1)!}\nabla^{p+1}f(\eta_{\x_j})\cdot(\x_j)^{\otimes p+1}\right|+\left|\int_{\Omega_0}\frac{1}{(p+1)!}\nabla^{p+1}f(\eta_\x)\cdot(\x)^{\otimes p+1}\,d\x\right|\\
\leq & \sum\limits_{j=1}^J\frac{1}{(p+1)!}|\nabla^{p+1}f(\eta_{\x_j})|\,|w_j|\,|\x_j|^{p+1}+\int_{\Omega_0}\frac{1}{(p+1)!}|\nabla^{p+1}f(\eta_\x)|\,|\x|^{p+1}\,d\x\\
\leq & \frac{1}{(p+1)!}||\nabla f^{p+1}||_\infty\left(\text{diam}(\Omega_0)^{p+1}|\Omega_0|+\int_{\Omega_0}|\x|^{p+1}\,d\x\right). 
\end{split}
\end{equation}
In the above, we have used that $w_j\geq 0$, and as the rule is of order $p$, $\sum\limits_{j=1}^J w_j=|\Omega_0|$. As $\Omega_0$ is bounded, the term in brackets is finite, yielding the result. 
\end{proof}

\begin{proposition}
Let $\Omega_0$ be a bounded domain and $Q_M:C(\overline{\Omega}_0)\to\mathbb{R}$ be a quadrature rule, 
$$
Q_M(f)=\sum\limits_{j=1}^J f(\x_j)w_j
$$
with $\x_j\in \Omega_0$ such that
$$
\left|Q_M(f)-\int_{\Omega_0}f(\x)\,d\x\right|\lesssim ||\nabla^{p+1}f||_\infty 
$$
for all $f\in C^{p+1}(\bar{\Omega}_0)$. If $A_n\in \mathbb{R}^{d\times d}$ has positive determinant with $\text{cond}(A_n)=||A_n^{-1}||\,||A_n||\lesssim 1$, and $\y_n\in\mathbb{R}^d$, $E_n=A_n\Omega_0+\y_n$, and $Q_n:C(\overline{E}_n)\to\mathbb{R}$ is defined by 
\begin{equation}\label{eqAppendixDefElementInt}
Q_n(f)=\sum\limits_{j=1}^J f(A_{\edit{n}}\edit{\x_{n,j}}+\y_n)\det(A)\edit{w_{n,j}},
\end{equation}
then 
\begin{equation}
\left|Q_n(f)-\int_E f(\x)\,d\x\right|\lesssim |E_n|\,\text{diam}(E_n)^{p+1}||\nabla^{p+1}f||_\infty
\end{equation}
for all $f\in C^{p+1}(\overline{E}_n)$.

\end{proposition}\label{propElementEstimate}
\begin{proof}
Define $\tilde{f}:\Omega_0\to\mathbb{R}$ by $\tilde{f}(\x)=f(A_n\x+\y_n)$. If $f\in C^{p+1}(\overline{E}_n)$, then $\tilde{f}\in C^{p+1}(\overline{\Omega}_0)$. Thus, employing the change of variables $\x'=A_n^{-1}\x+\y_n$, 
\begin{equation}
\begin{split}
\left|Q_n(f)-\int_{E_n} f(\x)\,d\x\right|=&\left|\sum\limits_{j=1}^J f(A_n\edit{\x_{n,j}}+\y_n)\det(A_n)\edit{w_{n,j}}-\det(A_n)\int_{\Omega_0}\tilde{f}(\x')\,d\x'\right|\\
= & \left|\sum\limits_{j=1}^J \tilde{f}(\edit{\x_{n,j}})\det(A_n)\edit{w_{n,j}}-\det(A_n)\int_{\Omega_0}\tilde{f}(\x')\,d\x'\right|\\
\lesssim & \det(A_n)||\nabla^{p+1}\tilde{f}||_\infty. 
\end{split}
\end{equation}
We note that $|E_n|=\det(A_n)|\Omega_0|$. Furthermore, we have the estimate $||\nabla^{p+1}\tilde{f}||_\infty\lesssim ||A_n||^{p+1}||\nabla^{p+1}f||_\infty$, which we aim to convert into an estimation based on the geometry of $E_n$, rather than the operator $A_n$. This is obtained by first observing that
\begin{equation}
\text{diam}(\Omega_0)=\sup\limits_{\x,\y\in \Omega_0}|\x-\y|=\sup\limits_{\x,\y\in E_n}|A_n^{-1}\x-A_n^{-1}\y|\leq  ||A_n^{-1}||\sup\limits_{\x,\y\in E_n}|\x-\y|=||A_n^{-1}||\,\text{diam}(E_n),
\end{equation}
which may be rearranged to give 
$$
||A_n^{-1}||^{-1}\leq \frac{\text{diam}(E_n)}{\text{diam}(\Omega_0)}. 
$$
As $\text{cond}(A_n)\lesssim 1$, we have that $||A_n||\lesssim ||A_n^{-1}||^{-1}$, and thus 
\begin{equation}
\begin{split}
||A_n||\lesssim ||A_n^{-1}||^{-1}\leq &\frac{\text{diam}(E_n)}{\text{diam}(\Omega_0)}\lesssim \text{diam}(E_n).
\end{split}
\end{equation}
Combining these final estimates yields
\begin{equation}
\left|Q_n(f)-\int_{E_n} f(\x)\,d\x\right|\lesssim |E_n|\,\text{diam}(E_n)^{p+1}||\nabla^{p+1}f||_\infty.
\end{equation}
\end{proof}

\begin{proposition}\label{propVarUpperBound}
Let $(E_n)_{n=1}^{N_e}$ be a partition of $\Omega$ with $E_n=A_n\Omega_0 +\y_n$. We assume that there is a length scale $h$ such that, independently of $n$ and $N_e$, $\text{diam}(E_n)\lesssim h$, $|E_n|\lesssim h^d$, $||A_n^{-1}||\,||A_n||\lesssim 1$, and $h^d\lesssim N=JN_e\lesssim h^d$.  

Let $Q_M$ be a quadrature rule with points and weights $\xi$ on the master element $\Omega_0$ such that, for every $f\in C^{p+1}(\overline{\Omega}_0)$, 
\begin{equation}
\left|Q_M(f;\xi)-\int_{\Omega_0}f(\x)\,dx\right|\lesssim||\nabla^{p+1}f||_\infty.
\end{equation}
If the quadrature rule is extended via $N_e$ (not necessarily independent) quadrature weights and points $(\xi_n)_{n=1}^{N_e}$ to a quadrature rule on $\Omega$ via
 $$
 Q(f;(\xi_n)_{n=1}^{N_e})=\sum\limits_{n=1}^{N_e}Q_n(f;\xi_n)
 $$
 with $Q_n$ as in \eqref{eqAppendixDefElementInt}, then 
 $$
\left|Q(f;(\xi_n)_{n=1}^{N_e})-\int_\Omega f(\x)\,d\x\right|\lesssim N^{-\frac{p+1}{d}}||\nabla^{p+1}f||_\infty.
 $$
 Furthermore, if the quadrature points are i.i.d. sampled from a distribution $\mu$ and yield an unbiased rule, the variance estimate
$$
\text{Var}(Q(f;(\xi_n)_{n=1}^{N_e})\lesssim N^{-1-\frac{2p+2}{d}}||\nabla^{p+1}f||_{\infty}^2
$$
holds.
\end{proposition}
\begin{proof}

Under the given assumptions, the uniform estimate from \Cref{propElementEstimate} becomes 
\begin{equation}
\begin{split}
\left|Q_n(f;\xi_n)-\int_{E_n}f(\x)\,d\x\right|\lesssim h^{d+p+1}||\nabla^{p+1}f||_\infty. 
\end{split}
\end{equation}

The uniform estimate arises by summing over elements, as 
\begin{equation}
\begin{split}
\left|Q(f;(\xi_n)_{n=1}^{N_e})-\int_{\Omega}f(\x)\,d\x\right|=& \left|\sum\limits_{n=1}^{N_e}Q_n(f;\xi_n)-\int_{E_n}f(\x)\,d\x\right|\\
\leq & \sum\limits_{n=1}^{N_e}\left|Q_n(f;\xi_n)-\int_{E_n}f(\x)\,d\x\right|\\
\lesssim &N_e h^{d+p+1}||\nabla f||_\infty\lesssim N^{-\frac{p+1}{d}}||\nabla^{p+1}f||_\infty. 
\end{split}
\end{equation}

As each element-wise quadrature rule is statistically independent, we may decompose the variance of the global quadrature rule as
\begin{equation}
\begin{split}
\text{Var}(Q(f;(\xi_n)_{n=1}^{N_e})=&\text{Var}\left(\sum\limits_{n=1}^{N_e}Q_n(f;\xi_n)\right)\\
=& \sum\limits_{n=1}^{N_e}\text{Var}\left(Q_n(f;\xi_n)\right)\\
=& \sum\limits_{n=1}^{N_e}\mathbb{E}\left(\left|Q_n(f;\xi_n)-\int_{E_n}f(\x)\,d\x\right|^2\right)\\
\lesssim & N_e h^{2d+2p+2}||\nabla^{p+1}f||_\infty^2\lesssim N^{-1-\frac{2p+2}{d}}.
\end{split}
\end{equation}

\end{proof}

\end{document}